\documentclass{article}
\usepackage[utf8]{inputenc}
\usepackage{fullpage}
\usepackage{hyperref}
\usepackage{xcolor}
\usepackage{amsmath,amsfonts,amssymb}
\usepackage{graphicx}%
\usepackage{subcaption}
\usepackage{mathrsfs}
\usepackage{subcaption}
\usepackage[margin=1.0in]{geometry}
\usepackage{appendix}
\usepackage{accents}

\usepackage{accents}

\usepackage{amsthm}
\usepackage[shortlabels]{enumitem}
\newtheorem{theorem}{Theorem}[section]
\newtheorem{remark}[theorem]{Remark}
\newtheorem{proposition}[theorem]{Proposition}
\newtheorem{hypothesis}[theorem]{Hypothesis}

\newtheorem{lemma}[theorem]{Lemma}
\newtheorem{corollary}[theorem]{Corollary}

\newcommand{\sff}{{\mathsf f}}
\newcommand{\sfg}{{\mathsf g}}
\newcommand{\sfS}{{\mathsf S}}
\newcommand{\sfT}{{\mathsf T}}
\newcommand{\sfx}{{\mathsf x}}
\newcommand{\sfR}{{\mathsf R}}
\newcommand{\sfB}{{\mathsf B}}
\newcommand{\Sp}{S_p}
\newcommand{\SpI}{S_p^I}

\newcommand{\mQ}{\textcolor{black}{\mathcal{Q}_z}}
\newcommand{\eps}{\varepsilon}

\newcommand{\Rm}{\mathbb R}
\newcommand{\dint}{\displaystyle\int}

\newcommand{\tps}{\mathrm{t}}

\newcommand{\sL}{{\mathsf L}}

\title{Long distance propagation of light in random media with partially coherent sources}
\author{Guillaume Bal \thanks{Departments of Statistics and Mathematics and Committee on Computational and Applied Mathematics, University of Chicago, Chicago, IL 60637; guillaumebal@uchicago.edu} \and Anjali Nair \thanks{Department of Statistics and Committee on Computational and Applied Mathematics, University of Chicago, Chicago, IL 60637; anjalinair@uchicago.edu}}
\date{\today}

\begin{document}

\maketitle

\begin{abstract}
 Optical beam propagation in random media is characterized by familiar speckle patterns generated by intricate interference effects. Such patterns may be modified and possibly attenuated for partially coherent incident beam profiles. In the weak-coupling regime of the It\^o-Schr\"odinger paraxial model of wave propagation, we show how the spatio-temporal statistics of the partially coherent beams interact with the statistics of the random medium to enhance or suppress scintillation effects.
\end{abstract}


\section{Introduction}
\label{sec:intro}

It is widely acknowledged in physical literature that incident beams with randomly varying spatial and/or temporal fluctuations, i.e, partially coherent sources~\cite{mandel1995optical} are potentially more robust against turbulent fluctuations  when compared to their fully coherent counterparts~\cite{gbur2014partially}. This fact has also been supported by several physical experiments~\cite{dogariu2003propagation, nelson2016scintillation}, numerical simulations~\cite{wu1991coherence, borah2010spatially, nair2023scintillation} and formal perturbative expansions~\cite{gbur2002spreading, korotkova2004model}. 

A primary metric used to quantify statistical stability is the scintillation index, which is the ratio of the variance of intensity and the average intensity squared. Under long distances of propagation through random media, it is a well accepted conjecture that the wave field becomes circularly symmetric Gaussian distributed and the scintillation index converges to one~\cite{goodman1976some, sheng1990scattering}. This provides a model for speckle patterns observed in experiments~\cite{goodman1976some, andrews2001laser, carminati2021principles}. This conjecture is not true in general, for instance the spot dancing regime leads to a Rice-Nakagami distribution for deterministic Gaussian sources~\cite{furutsu1973spot, dawson1984random}. 

The Gaussian conjecture was recently demonstrated in the weak-coupling (scintillation) regime of a paraxial wave propagation model in \cite{bal2024complex} with a complete statistical description of the wave field for deterministic incident beams. The analysis of second and fourth statistical moments leading to a full characterization of the scintillation index was carried out in  \cite{garnier2014scintillation, garnier2016fourth}. In particular, it was shown  in~\cite{bal2024complex} that under the diffusive scaling of the scintillation regime, the wave field indeed becomes circularly symmetric Gaussian. In such a diffusion regime, the average intensity follows a diffusion equation with an anomalous rate of $z^{3/2}$ with $z$ distance of propagation along the beam direction.

However the presence of random sources leads to more complex mixing regimes with scintillation index different from unity~\textcolor{black}{\cite{andrews2005laser,berman2009reduction, berman2009influence, efimov2014simultaneous, garnier2022scintillation}}. It has been observed that the superior performance of such sources towards scintillation reduction is in part due to temporal averaging at the detector. If the averaging time of the detector is much smaller than the decorrelation time of the source, scintillation can potentially be higher for partially coherent sources due to the additional randomness of the source~\cite{berman2009influence, gbur2014partially, garnier2022scintillation}. The goal of this paper is to provide a  description of the limiting distribution of the wave field due to partially coherent sources under the kinetic and diffusive scalings of the scintillation regime. We will also use this to characterize the long distance statistics of time averaged intensity in terms of the coherence length scales of the source and time averaging at the detector.

While wave propagation is a classical phenomenon, a complete mathematical description of waves in random media is still in lack and various scaling regimes have been identified to provide a macroscopic description of the underlying wave field \cite{BKR-KRM-10}. In particular, the paraxial approximation is commonly used to model long distance propagation of high frequency waves, where the wave field retains its beam-like structure along a privileged axis with negligible backscattering~\cite{andrews2001laser}.

A starting point for modelling time harmonic scalar wave fields is the scalar Helmholtz equation
\[
   \Big( \partial^2_z + \Delta_x + k_0^2n^2(z,x)\Big) p(z,x) = \textcolor{black}{+}u_0(x) \textcolor{black}{\delta'_0(z)}\,,
\]
where $k_0$ is the carrier wave number, $n$ is the refractive index of the medium and $u_0$ denotes the source. Throughout the paper $z>0$ models distance of propagation along the beam axis while $x\in\Rm^d$ describes transverse spatial variables. While $d=2$ is the physical dimension, our results apply for any $d\geq1$.

Let $n^2(z,x)=1+\nu(z,x)$ with zero mean random medium fluctuations $\nu(z,x)$ satisfying $\mathbb{E}[\nu(0,0)^2]=\sigma^2$. Assuming that the fluctuations in the refractive index are small, we consider a high frequency, long distance scaling called the white noise paraxial regime:
\[
  z \to \frac{z}{\vartheta},\qquad x \to x,\qquad k_0\to \frac{k_0}\vartheta,\qquad \sigma^2\to \vartheta^3 \sigma^2, \qquad u_0\to u_0,
\]
where $\vartheta\ll1$ is a non dimensionalised parameter. Under this scaling, the paraxial envelope of the wavefield $u(z,x)=p\big(\frac{z}{\vartheta},x\big)e^{-i\frac{k_0z}{\vartheta^2}}$ follows the paraxial wave equation
\[
  \Big(2ik_0 \partial_z + \Delta_x + \frac {k_0^2}{\vartheta^{\frac12}} \nu(\frac z\vartheta,x) \Big)u = 0,\qquad u(0,x)=u_0(x)
\]
where formally, the term generating backscattering is neglected assuming $|\vartheta^2\partial^2_{z}u|\ll 1$. When the correlations in the medium are sufficiently short range, a formal application of the central limit theorem gives that
\[
\frac {1}{\vartheta^{\frac12}} \nu \big(\frac {z}\vartheta,x \big)\mathrm{d}z\approx B(\mathrm{d}z,x)\,,
\]
with 
\[
\mathbb{E}[B(z,x)B(z',x')]=\min(z,z')R(x-x')\,,
\]
where the spatial covariance $R(x)=2\int\limits_{0}^\infty\mathbb{E}[\nu(0,0)\nu(z,x)]\mathrm{d}z$. This gives in the limit $\vartheta\to 0$ the It\^o-Schr\"odinger equation
\begin{equation*}
\mathrm{d}u(z,x)=\frac{i}{2k_0}\Delta_xu(z,x)\mathrm{d}z-\frac{k_0^2}{8}R(0)u(z,x)\mathrm{d}z+\frac{ik_0}{2}u(z,x)\mathrm{d}B(z,x),\quad u(0,x)=u_0(x)\,.
\end{equation*}
A rigorous derivation of the white noise paraxial limit, which is our starting point in this paper, from the Helmholtz model can be found in~\cite{bailly1996parabolic, garnier2009coupled}. 

The main advantage of the It\^o-Schr\"odinger equation is the availability of closed form partial differential equations for moments of all order of the wave field. However, PDEs for moments higher than two cannot be solved analytically in general, which necessitates identifying additional scaling regimes amenable to asymptotic analysis. One such regime is the scintillation regime, which is a weak fluctuation regime where the effects of turbulence are at the scale $\eps_{\rm eff}\ll 1$ while the total propagation length is sufficiently large so that $z_{\rm eff}\eps_{\rm eff}\gtrsim 1$. As in~\cite{bal2024complex}, under the scintillation regime we identify two scaling regimes, the first where $z_{\rm eff}\eps_{\rm eff}\approx 1$ referred to as the kinetic regime and a second $z_{\rm eff}\eps_{\rm eff}\gg 1$ which we call the diffusive regime. 

We may assume in practice that the time it takes the signal to reach a receiver at a distance $z=L$ is much smaller than the time scale of fluctuations of both the medium and the source. This allows us to introduce a slowly varying time parameter into the wave field as $u(z,x;t)$ describing a randomly varying source $u_0(x;t)$. While the medium statistics could be a function of time as well, i.e, $B=B(z,x;t)$, we assume $t$ above to be frozen during the beam propagation in the sense that the decorrelation time of the medium $(\tau_m)$ is much slower than the decorrelation time of the source $(\tau_s)$. We also assume the latter to be comparable to the averaging time at the detector $(\sfT)$, i.e, $\tau_s\le \sfT\ll \tau_m$. The framework presented here can be generalized to include an explicit temporal dependence on the medium as well. The assumption $\sfT>0$ models the fact that light detectors are typically not sufficiently sensitive to capture the rapid temporal oscillations of laser light. We will see that this inability for detectors to measure such rapid oscillations is one of the two ingredients leading to reduced scintillation. The second ingredient is an assumption that the spatial correlation of the random source is appropriately small compared to the width of the incident beam.

The paper is organized as follows. After introducing our model of partially coherent sources in section \ref{sec:description}, we state the main results in section~\ref{sec:main_results}. Some interpretations and illustrations of these results are proposed in section \ref{sec:illustration}. The proofs of the results mainly follow from identifying the limiting moments of all order of the wave field. Conditioned on the realization of the random incident beam, the analysis of such moments was carried out in \cite{bal2024complex}. The main objective of this paper is to show how such moments are modified in the presence of random incident beams.
In Section~\ref{sec:second_moments}, we derive the limiting second moments in the kinetic and diffusive regimes. The higher moments are characterized in Section~\ref{sec:higher_moments}. We will use this result to obtain the limiting moments of the time averaged intensity in Section~\ref{sec:limiting_intensity}. Finally in Section~\ref{sec:conclusion} we collect some concluding remarks and possible extensions. An appendix presents the proof of the tightness and stochastic continuity results given in Theorem \ref{thm:tightness} below.
\section{Description of the partially coherent wavefield.}
\label{sec:description}
The main goal of this paper is to examine the long distance statistics of wave fields propagating in randomly varying media with partially coherent sources. Our starting point is the It\^{o}-Schr\"{o}dinger equation for the complex valued wave field $u(z,x;t)$ given on a domain $(z,x)\in[0,\infty)\times \mathbb{R}^d, t\in[0,\infty)$ by~\cite{garnier2022scintillation}
\begin{equation}\label{eqn:Ito_Schro_1}
\mathrm{d}u(z,x;t)=\frac{i}{2k_0}\Delta_xu(z,x;t)\mathrm{d}z-\frac{k_0^2}{8}R(0)u(z,x;t)\mathrm{d}z+\frac{ik_0}{2}u(z,x;t)\mathrm{d}B(z,x;t) ,
\end{equation}
where $t$ reflects the time scale of the random fluctuations in the incident beam. We assume that this time scale is much larger than the time it takes an incident beam to reach a receiver at distance $z>0$. Thus, $t$ can be assumed to be a frozen parameter in~\eqref{eqn:Ito_Schro_1}. The randomness in the medium is encoded in $B(z,x;t)$, a Gaussian process over $[0,\infty)\times\mathbb{R}^d\times [0,\infty)$ with statistics
\begin{equation}
    \mathbb{E}[B(z,x;t)]B(z',x';t')]=(z\wedge z')g(t-t')R(x-x')\,.
\end{equation}
As noted in the introduction, we assume that the decorrelation time of the random medium ($\tau_m$) is much slower than the decorrelation time of the source $\tau_s$ and an intrinsic time scale of the detector $\sfT$,  i.e. $\tau_m\gg \tau_s$ and $\tau_m\gg \sfT$. As a consequence, $g(t-t')\approx g(0)\equiv1$ in the rest of the paper. Fluctuations in time of the medium may be included as in \cite{garnier2022scintillation} when detectors average signals over times $\sfT\approx \tau_m$ but we will not do so here for simplicity. Further assumptions on the spatial covariance $R$ are listed below in Hypothesis~\ref{hyp:medium}. 

The source $u(0,x;t)=u_0(x;t)$ is assumed to be a mean zero complex Gaussian process over $\mathbb{R}^d\times [0,\infty)$, independent of $B$. Its properties are detailed in the next paragraph.
 
As in~\cite{bal2024complex}, we introduce the scaling 
\begin{equation*}
    z\to\frac{\eta z}{\eps},\quad R^\eps(x)=\frac{\eps}{\eta^3}R(x)\,,
\end{equation*}
so that \eqref{eqn:Ito_Schro_1} becomes
\begin{equation}\label{eqn:Ito_Schr}
\mathrm{d}u^\eps(z,x;t)=\frac{i\eta}{2k_0\eps}\Delta_xu^\eps(z,x;t)\mathrm{d}z-\frac{k_0^2}{8\eta^2}R(0)u^\eps(z,x;t)\mathrm{d}z+\frac{ik_0}{2\eta}u^\eps(z,x;t)\mathrm{d}B(z,x),
\end{equation}
with incident beam profile $u^\eps(0,x;t)=u^\eps_0(x;t)$. 

Here, $\eps\ll1$ and $\eta=\eta(\eps)$, for which we will consider two cases: one with $\eta=1$ referred to as the kinetic regime in which turbulence has an $O(1)$ effect on the wavefield since fluctuations of order $O(\eps)$ are integrated over distances of order $O(\eps^{-1})$. In the second regime, $\eta=(\ln|\ln\eps|)^{-1}$, where the cumulative effects of turbulence is $O(\eta^{-2})\gg O(1)$ since fluctuations of order $O(\eps\eta^{-3})$ are integrated over distances of order $O(\eta\eps^{-1})$. We call this latter scaling as the diffusive regime of the white noise paraxial equation. The large separation of scales in the choice of $\eta^{-1}=\ln\ln\eps^{-1}$ is made for technical reasons \cite{bal2024complex}.

\paragraph{Scaling of the source: }
We assume that the random incident beam follows a mean zero complex Gaussian process with the following statistics:
\begin{equation}\label{eqn:u0_cov}
    \mathbb{E}[u_0^\eps(x;t){u_0^\eps}^\ast(y;t')]=F\Big(\frac{t-t'}{\tau_s}\Big)J^\eps(x,y)\,,
\end{equation}
with $F(-t)=F(t)\in\sL^1(\mathbb{R})\cap\sL^\infty(\mathbb{R})$ and \textcolor{black}{for simplicity, we assume $0\le |F(t)|\le F(0)=1$} and $0<\tau_s\leq1$ models the time scale of the random source. To simplify notation, we will simply replace the above term by \textcolor{black}{$F(t-t')$}. We will observe that only for $\tau_s$ sufficiently small compared to the detector time scale $\sfT$ do we expect a small scintillation index.

We define the spatial two point correlation function as
\begin{equation*}
  \Gamma^\eps(r,\sigma)=  J^\eps \big(r+ \frac12 \sigma,r- \frac12 \sigma\big)\,,
\end{equation*}
and assume a scaling of the form
\begin{eqnarray}\label{eqn:Gamma_eps}
    \Gamma^\eps(r,\sigma) =\Gamma\Big(\eps^{\beta}r,\frac{\eps^{\beta}}{\theta}\sigma\Big).
\end{eqnarray}
\textcolor{black}{We assume that $\Gamma$ depends only on the barrycentre $r=\frac{x+y}{2}$ and the magnitude of separation between the points $\sigma=x-y$ so that the covariance $J(x,y)=\Gamma\big(\frac{x+y}{2},x-y\big)=\Gamma\big(\frac{x+y}{2},|x-y|\big)$.} 
Here, $\Gamma(r,-\sigma)=\Gamma(r,\sigma)\in\mathcal{S}(\mathbb{R}^{2d})$ and $0<\theta\leq 1$ a parameter (independent of $\eps$) that controls the spatial correlation length of the source with respect to the overall width of the beam. As in \cite{bal2024complex}, the statistics of the wavefield satisfy an asymptotic limit in the kinetic and diffusive regimes as $\eps\to0$ provided that the incident beam is sufficiently broad. This is modeled by a choice of parameter $\beta\ge 1$ with the incident field a broad beam given as a (smooth) function of $\eps^\beta x$ and hence a diameter of order $\eps^{-\beta}$. The case $\beta=1$ generates a richer limiting diffusive model. The case $\beta>1$ models very broad incident beams as a good approximation to plane waves. 

The parameter $\theta$ models the spatial fluctuations of the incident beam. The case $\theta=1$ corresponds to a spatial correlation of the incident beam that is as broad as the incident beam. In such a setting, we do not expect much averaging effects coming from the randomness of the partially coherent source. The setting $\theta\ll 1$ indicates a coherence length much smaller than the overall scale of the beam in which we will show that averaging effects are more pronounced. Only in the limit where $\theta\ll1$ {\em and} $\tau_s\ll \sfT$ do we expect asymptotically vanishing scintillation effects. 

\medskip

We denote by $\mathcal{M}_B(\mathbb{R}^{nd})$, the Banach space of finite signed measures on $\mathbb{R}^{nd}$ for $n$ an integer, equipped with the total variation (TV) norm $\|\cdot\|$; for measures with a density with respect to the Lebesgue measure, then that norm is the $\sL^1(\mathbb{R}^{nd})$ norm of that density. Throughout the paper, $\|\cdot\|$ is used for the TV norm of bounded Radon measures or for the operator norm of operators acting on such measures.

\medskip

We denote by $\hat{\Gamma}$, the Fourier transform of $\Gamma$ given by
\begin{equation*}
    \hat{\Gamma}(\zeta,\xi)=\int\limits_{\mathbb{R}^{2d}}\Gamma(r,\sigma)e^{-i(r\cdot\zeta-\sigma\cdot\xi)}\mathrm{d}r\mathrm{d}\sigma\,
\end{equation*}
with $\hat{\Gamma}\in\mathcal{S}(\mathbb{R}^{2d})$. This means that the Fourier transform of~\eqref{eqn:Gamma_eps}
\begin{eqnarray}
    \hat{\Gamma}^\eps(\zeta,\xi)=\int\limits_{\mathbb{R}^{2d}}\Gamma^\eps(r,\sigma)e^{-i(\zeta\cdot r-\xi\cdot\sigma)}\mathrm{d}r\mathrm{d}\sigma=\theta^d\eps^{-2\beta d}\hat{\Gamma}(\eps^{-\beta}\zeta,\theta\eps^{-\beta}\xi)
\end{eqnarray}
is the density of a measure $\hat\Gamma^\eps d\zeta d\xi$ (still called $\hat\Gamma^\eps$) in $\mathcal{M}_B(\mathbb{R}^{2d})$  with total variation $\|\hat{\Gamma}^\eps\|=\|\hat{\Gamma}\|\le C$. This also means that $\hat{J}^\eps$, the Fourier transform of $J^\eps$, is integrable with the same bound. 

As in \cite{bal2024complex}, the analysis of the wavefields in the kinetic and diffusive regimes as $\eps\to0$ will be based on a careful analysis of the statistical moments of $u^\eps(z,x;t)$. Due to the complex circular Gaussian statistics in \eqref{eqn:u0_cov}, the $p+q$th moment of the process modeling the incident beam is given by
\begin{eqnarray}\label{eqn:u0_complex_Gaussian}
    \begin{aligned}
        \mathbb{E}[\prod\limits_{j=1}^pu_0^\eps(x_j;t_j)\prod\limits_{l=1}^qu_0^{\eps *}(y_l;t'_l)]=\begin{cases}
            0,\quad p\neq q\\
            \sum\limits_{\pi_p}\prod\limits_{j=1}^pF(t_j-t'_{\pi_p(j)})J^\eps(x_j,y_{\pi_p(j)}),\quad p=q\,,
        \end{cases}
    \end{aligned}
\end{eqnarray}
for a sum that runs over all permutations $\pi_p$ of $p$ integers~\cite{reed1962moment, mandel1995optical}. Moreover, using the simplifying notation  $v=(\xi_1,\ldots,\xi_p,\zeta_1,\ldots,\zeta_p)$ and $T=(t_1,\ldots,t_p,t'_1,\ldots,t'_p)$, the Fourier transform of the $p+p$th moment of the source is integrable, with $\|\hat{\mu}^\eps_{p,p}(0,v;T)\|\le p!\|F\|_\infty^p\|\hat{\Gamma}\|^p$. The $p+q$th moment for $p\neq q$ is zero identically. We use here $\|\cdot\|_\infty$ for the $\sL^\infty$ norm of a bounded (measurable) function. 

In the limit as $\eps\to0$, the relevant statistical information on the incident beam to describe the properties of the random field after large-distance propagation is captured by
\begin{equation}\label{eqn:Sp}
        \Sp(X,Y;T) :=\sum\limits_{\pi_p}\prod\limits_{j=1}^pF(t_j-t'_{\pi_p(j)})J_\theta(x_j,y_{\pi_p(j)}), \qquad \SpI(X):= \Sp(X,X;0)
\end{equation}
with $J_\theta$ defined as
\begin{equation}\label{eqn:J_theta}
    J_\theta(x,y)=\Gamma\Big(\frac{x+y}{2},\frac{x-y}{\theta}\Big)\,.
\end{equation}

Our derivation requires natural smoothness assumptions on the lateral spatial covariance of the random medium, which we summarize in the following hypothesis:
\begin{hypothesis}{(Assumptions on the covariance $R$) : }\label{hyp:medium}
    We assume that the spatial covariance $R(-x)=R(x)\in \sL^1(\mathbb{R}^d)\cap \sL^\infty(\mathbb{R}^d)$ which implies that the Fourier transform $\hat{R}(\xi)\in \sL^1(\mathbb{R}^d)\cap \sL^\infty(\mathbb{R}^d)$ as well. Since $\hat{R}$ is a power spectrum, we also have that $\hat{R}\ge 0$. We assume that $\hat{R}(\xi)\le \hat{\sfR}(\xi)=\hat{\sfR}(|\xi|)$ for some radially symmetric $\hat{\sfR}\in  \sL^1(\mathbb{R}^d)$ and $R(\tau + se)$ is integrable in $s$ for any $\tau\in\mathbb{R}^{d}$ and $e\in\mathbb{S}^{d-1}$. When the lateral dimension $d\ge 3$, we assume that $\langle\xi\rangle^{d-2}\hat{R}(\xi)\in\sL^\infty(\mathbb{R}^d)$, where $\langle \xi\rangle=\sqrt{1+|\xi|^2}$. 
    
    Furthermore in the diffusive regime, we require that $R(x)$ is maximal at $x=0$ and sufficiently smooth so that the Hessian 
\begin{equation}\label{eq:Hessian_Xi}
    \Xi:=\nabla^2 R(0)
\end{equation}
is defined and negative definite. 
\end{hypothesis}
In the diffusive regime with $\eta^{-1}=\ln\ln\eps^{-1}$, only the tensor $\Xi$ appears in the limiting models for the wavefield as $\eps\to0$.

\paragraph{Spatially rescaled random vector:} We finally consider the random process
\begin{eqnarray}
    \phi^\eps(z,r,x;t)=u^\eps(z,\eps^{-\beta} r+\eta x;t).
\end{eqnarray}
This process is rescaled so that it admits a non-trivial limit as $\eps\to0$ both in the kinetic and diffusive regimes.
The statistical moments of $\phi^\eps$ are denoted by 
\begin{eqnarray}
    m_{p,q}^\eps(z,r,X,Y;T)=\mathbb{E}[\prod_{j=1}^p\phi^\eps(z,r,x_j;t_j)\prod_{l=1}^q\phi^{\eps *}(z,r,y_l;t'_l)],
\end{eqnarray}
where $(X,Y)=(x_1,\ldots,x_p,y_1,\ldots,y_q)$ and $T=(t_1,\cdots,t_p,t'_1,\cdots,t'_q)$. For a collection of points $X=(x_1,\cdots,x_N)\in\mathbb{R}^{Nd}$ and $\tilde{T}=(t_1,\cdots,t_N)\in[0,\infty)^{N}$, we can also define a vector of such processes
\begin{eqnarray}
    \Phi^\eps(z,r,X;\tilde{T})=\big(\phi^\eps(z,r,x_1;t_1),\cdots,\phi^\eps(z,r,x_N;t_N)\big).
\end{eqnarray}
We may define the $p+q$th moments of this random vector as
\begin{eqnarray}\label{eqn:m_pp_random_vec}
    m_{p,q}^\eps(z,r,X_s,Y_s;T_s)=\mathbb{E}[\prod_{j=1}^p\phi^\eps(z,r,x_{s_j};t_{s_j})\prod_{l=1}^q\phi^{\eps *}(z,r,x_{s'_l};t'_{s'_l})]\,,
\end{eqnarray}
where $X_s=(x_{s_1},\cdots,x_{s_p}), Y_s=(x_{s'_1},\cdots,x_{s'_q})$, $T_s=(t_{s_1},\cdots,t_{s_p},t_{s'_1},\cdots,t_{s'_q})$ and $\{s_1,\cdots,s_p\}$ and $\{s'_1,\cdots,s'_q\}$ are integers drawn from $\{1,\cdots,N\}$ with replacement. We define the $p+q$th moments $M_{p,q}$ of the limiting random vector $\Phi$ in the same manner.

\paragraph{Statistical moment analysis and complex Gaussianity of coherent beams.}

One main technical advantage of the white noise model~\eqref{eqn:Ito_Schr} is that the following statistical moments of $u^\eps$ satisfy closed form partial differential equations:
 \begin{equation}\label{eq:muepspq}
\mu^\eps_{p,q}(z,X,Y;T)=\mathbb{E}[\prod\limits_{j=1}^pu^\eps(z,x_j;t_j)\prod\limits_{l=1}^q{u^\eps}^\ast(z,y_l;t'_l)]\,.
\end{equation}
These partial differential equations will be presented later in the text. They are central in the wavefield analysis of \cite{bal2024complex} and were used to analyze moments up to fourth order in~\cite{garnier2014scintillation, garnier2016fourth, garnier2018noninvasive, garnier2022scintillation, garnier2023fourth}. 

Under appropriate smoothness assumptions on the deterministic source, the analysis of such moments was used in \cite{bal2024complex} to demonstrate the circular Gaussian conjecture for broad incident beams in the diffusive regime. In particular, it was shown that for $\|\hat{\mu}_{p,q}^\eps(0)\|\le C^{p+q}$, the Fourier transform of $\mu^\eps_{p,q}$ admits the decomposition
\begin{equation}\label{eqn:mu_pq_Fourier_dec}
    \hat{\mu}^\eps_{p,q}(z,v)=\Pi^\eps_{p,q}(z,v)[N^\eps_{p,q}(z)\hat{\mu}^\eps_{p,q}(0)+E_{p,q}^\eps(z)\hat{\mu}^\eps_{p,q}(0)](v)\,,
\end{equation}
where \textcolor{black}{$v=(\xi_1,\cdots,\xi_p,\zeta_1,\cdots,\zeta_q)$ denotes the vector of Fourier variables dual to $(X,Y)$,} $N^\eps_{p,q}(z)$ and $E^\eps_{p,q}(z)$ are operators acting on $\mathcal{M}_B(\mathbb{R}^{(p+q)d})$ and $\Pi^\eps_{p,q}(z,v)$ is a unit magnitude complex exponential defined explicitly below in~\eqref{eqn:pi_pq}. $N^\eps_{p,q}(z)$ (defined in~\eqref{eqn:N_def}) is an operator constructed out of first and second moment type solution operators while the remainder satisfies $\|E_{p,q}^\eps(z)\|\le c(p,q,z)\eps^{\frac{1}{3}}$.  

The specific (Cartesian) product structure for coherent incident beams $\hat{\mu}^\eps_{p,q}(0)$ was then utilized in \cite{bal2024complex} to simplify the inverse Fourier transform, making it possible to conveniently express the $p+q$th moments in terms of combinations of first and second moments with a same error bounded in the uniform sense:
\begin{equation*}
  \begin{aligned}
      \mu_{p,q}^{\eps }(z,X,Y)&=\mathscr{F}(\mu^\eps_{1,0}(z, x_1),\ldots,\mu^\eps_{1,0}(z, x_p),\mu^\eps_{0,1}(z, y_1),\ldots,\mu^\eps_{0,1}(z, y_q),\mu^\eps_{1,1}(z, x_1,y_1),\ldots,\mu^\eps_{1,1}(z, x_p,y_q))
      \\&
      +O(\eps^{\frac{1}{3}})\,.
  \end{aligned}
\end{equation*}
Here, $\mathscr{F}$ is a continuous functional of $p+q+pq$ arguments defined in~\eqref{eqn:F_def}. This was used to prove a convergence in distribution of spatially rescaled random vectors $\Phi^\eps\boldsymbol{\Rightarrow}\Phi$ as defined above when $\eps\to 0$. In particular, in the diffusive regime $\Phi$ was indeed a circularly symmetric Gaussian random vector providing a proof of the Gaussian conjecture. The limiting intensity for such distributions then follows an exponential law with a correlation structure that is compatible with the speckle formation observed in laser beam experiments.

The main objective of this paper is to generalize approximations of the form \eqref{eqn:mu_pq_Fourier_dec} to the setting of partially coherent sources and to understand the interplay between the source randomness and the medium randomness in the kinetic and diffusive regimes as $\eps\to0$ as a function of the main decorrelation parameters of the source model $\tau_s$ and $\theta$.

For a fixed collection of points in time $T$, partially coherent sources as defined above still have moments in $\mathcal{M}_B(\mathbb{R}^{(p+q)d})$. This, along with the assumption that the medium is time independent, allows the decomposition~\eqref{eqn:mu_pq_Fourier_dec} to still hold. However, due to the lack of a Cartesian product structure in $\hat{\mu}^\eps_{p,q}(0,v;T)$, higher moments in general cannot be compactly represented through first and second moments in physical variables as above. In particular, the limiting intensity distribution of the spatially rescaled random vector $\phi^\eps(z,r,x;t)$ is not immediately evident. The additional randomness at the source potentially interacts with the randomness in the medium to generate non Gaussian statistics at the wave field, translating to intensity distributions that are no longer exponential.

\paragraph{Intensity and scintillation for partially coherent beams.}
The intensity $I^\eps$ of the rescaled process $\phi^\eps(z,r,x;t)$ is given by
\begin{eqnarray}
    I^\eps(z,r,x;t)=|\phi^\eps(z,r,x;t)|^2.
\end{eqnarray}
We expect, and will obtain, that such an intensity remains highly fluctuating in the limit $\eps\to0$, with fluctuations that are larger than in the setting of fully coherent (deterministic) incident beams since the random source introduces additional uncertainty.

However, a detector with an inherent time scale modeled by $\sfT$ typically cannot resolve the rapid oscillations at the scale $\tau_s$ and may be better modeled by the following averaged intensity of the rescaled process $\phi^\eps$:
\begin{equation}\label{eqn:intens_time_avg}
        I^\eps_{\sfT}(z,r,x;t)=\frac{1}{\sfT}\int\limits_{0}^\sfT I^\eps(z,r,x;t+t')\mathrm{d}t'.
\end{equation}
We have assumed here an integral modeled by a measure $\frac{1}{\sfT}\chi_{[0,\sfT]}(t)dt$, which could easily be replaced by a more general normalized weight $\nu_{\sfT}(t)dt$. We consider the above model for simplicity.

The statistical stability of this time averaged intensity can be quantified through the (time averaged) scintillation index
\begin{eqnarray}\label{eqn:scint_time_avg}
     \sfS_\sfT^\eps(z,r,x;t)=\frac{\mathbb{E}[I^\eps_\sfT(z,r,x;t)^2]-\mathbb{E}[I^\eps_\sfT(z,r,x;t)]^2}{\mathbb{E}[I^\eps_\sfT(z,r,x;t)]^2}\,.
\end{eqnarray}

For $\sfT$ large compared to $\tau_s$, we expect $I^\eps_{\sfT}$ to be essentially independent of $t$. This is the regime where we expect a significant reduction of the scintillation index.

\section{Main results}
\label{sec:main_results}

We now state the main results obtained in this paper. The first three results in Theorems \ref{thm:large_corr_kinetic} and \ref{thm:large_corr_diffusive} analyze the limits of finite dimensional distributions of the process $\phi^\eps$ in the kinetic and diffusive limits $\eps\to0$, respectively. Corollary \ref{coro:intensity_large_corr} then concerns the limit of the intensity $I^\eps$. Our main convergence result on the averaged intensity modeling detector readings is given in the diffusive regime in Theorem \ref{thm:time_avg_intens_limit} while
Corollary \ref{coro:scint_large_correl} presents the results for the scintillation index. The above results involve the analysis of finite dimensional distributions combined with an a priori compactness result obtained in our final Theorem \ref{thm:tightness}.

The statistics of the incident beam at $z=0$ that are relevant in the limits $\eps\to0$ are described in \eqref{eqn:Sp}. The full statistics of the random beam at larger values of $z$ involve several terms whose explicit expressions are not easily summarized and hence are presented in detail in later sections. Only $T-$dependent terms involve the statistics of the source and its scaling parameters $(\theta,\tau_s/\sfT)$.

\begin{theorem}[Kinetic regime]\label{thm:large_corr_kinetic} In the kinetic regime $\eta=1$, the mean zero random vector $\Phi^\eps\boldsymbol{\Rightarrow}\Phi$ in distribution as $\eps\to 0$ where $\Phi$ is a random vector with moments $M_{p,p}$ given by
\begin{equation}\label{eq:Mppkinetic}
  M_{p,p}(z,r,X,Y;T)=\begin{cases}
  \Sp(r,\ldots,r;T)\mathscr{F}\Big(M^\infty_{1,0}(z,r,x_j)_{j},M^\infty_{0,1}(z,r,y_l)_{l},M^\infty_{1,1}(z,r,x_j,y_l)_{j,l}\Big),\quad  \beta>1 \!\!\!\!\!\!\!\!\!\!\!\!\\
      \Big(M_{p,p}^\infty(z,\cdot,X,Y) \ast\Sp(\cdot;T)\Big)(r,\ldots,r),\quad \qquad  {\beta=1} \\ 
\mathscr{G}\big(M_{0,0}(z,r;t_j,t'_l)_{j,l},M_{1,1}(z,r,x_j,y_l;t_j,t'_l)_{j,l}\big),\quad \beta=1, \theta\to 0\,.
  \end{cases}
\end{equation}

In the case $\beta>1$, $M^\infty_{1,1}$ is given by~\eqref{eqn:M_11_inf_kinetic}, $\mathscr{F}$ is defined in~\eqref{eqn:F_def} and $M_{1,0}^\infty(z,r,x)=M_{0,1}^\infty(z,r,x)=e^{-\frac{k_0^2R(0)z}{8}}$. In the case $\beta=1$, $M_{p,p}^\infty$ is given by~\eqref{eqn:M_pp_inf_kinetic}, $\mathscr{G}$ is defined in~\eqref{eqn:G_def} and $M_{0,0}, M_{1,1}$ are given by~\eqref{eqn:M_00_kinetic} and~\eqref{eqn:M_11_kinetic} respectively. In all cases, $\Sp$ is given by~\eqref{eqn:Sp} and $M_{p,q}=0$ when $p\neq q$. 
\end{theorem}
\begin{theorem}[Diffusive regime]\label{thm:large_corr_diffusive} In the diffusive regime, the mean zero random vector $\Phi^\eps\boldsymbol{\Rightarrow}\Phi$ in distribution as $\eps\to 0$ where $\Phi$ is a random vector with moments given by
\begin{eqnarray}\label{eq:Mppdiffusive}
  M_{p,p}(z,r,X,Y;T)=\begin{cases}
  \Sp(r,\ldots,r;T) \sum\limits_{\pi_p}\prod\limits_{j=1}^pM_{1,1}^\infty(z,r,x_j,y_{\pi_p(j)}), \quad &\beta>1\\
   \Big(M_{p,p}^\infty(z,\cdot,X,Y)\ast\Sp(\cdot;T)\Big)(r,\ldots,r),\quad 
 & {\beta=1} 
 \\    
 \sum\limits_{\pi_p}\prod\limits_{j=1}^pM_{1,1}(z,r,x_{j},y_{\pi_p(j)};t_{j},t'_{\pi_p(j)}),\quad &\beta=1, \ \theta\to 0\,.
  \end{cases}
  \end{eqnarray}
Here, $M_{1,1}^\infty$ is given by~\eqref{eqn:M_11_inf_diff} when $\beta>1$ independent of $\theta$. $M_{p,p}^\infty$ is given by~\eqref{eqn:M_pp_inf_diff} and $M_{11}$ is given by~\eqref{eqn:M_11_diffusive} when $\beta=1$. In all cases, $\Sp$ is given by~\eqref{eqn:Sp} and $M_{p,q}=0$ when $p\neq q$.
\end{theorem}
\begin{corollary}[Intensity distribution]\label{coro:intensity_large_corr}
    In the diffusive regime, the intensity $I^\eps(z,r,x;t)=|\phi^\eps(z,r,x;t)|^2\boldsymbol{\Rightarrow} I(z,r)$ in distribution as $\eps\to 0$, where $I(z,r)$ is a random vector with moments given by
    \begin{eqnarray}\label{eqn:I}
    \mathbb{E}[I(z,r)^p]=\begin{cases}
    \Gamma(r,0)^pp!^2,\quad & \beta>1\\
          p![G_p(z^3,\cdot)\ast \SpI(\cdot)](r,\ldots,r),\quad & {\beta=1} \\ 
       p!\mathbb{E}[I](z,r)^p,\quad&\beta=1,\ \theta\to 0\,,
    \end{cases}
\end{eqnarray}
where $G_p$ follows a diffusion equation~\eqref{eqn:high_dim_diff} and the source $\SpI(X)$ is given by~\eqref{eqn:Sp}. In particular, the average intensity $\mathbb{E}[I](z,r)$ follows a diffusion equation given by~\eqref{eqn:E_I_large_correl_beta=1}, \eqref{eqn:G_large_correl_beta=1} with source $\Gamma(r,0)$  when $\beta=1$.
\end{corollary}

Let 
\begin{eqnarray}\label{eqn:F_p(T)}
F_{p}(\sfT)= \frac{1}{\sfT^p}\int\limits_{[0,\sfT]^p}\sum\limits_{\pi_p}\prod
_{j=1}^pF(t_j-t_{\pi_p(j)})^2\mathrm{d}t_1\ldots\mathrm{d}t_p\,,
\end{eqnarray}
with limiting cases
\begin{equation*}
    F_p(\sfT)\to\begin{cases}
            p!,\quad & \sfT\to 0\\
        1,\quad & \sfT\to \infty\,.
    \end{cases}
\end{equation*}
Then we have the following:
\begin{theorem}[Statistics of time averaged intensity]\label{thm:time_avg_intens_limit}In the diffusive regime, we have that
\begin{eqnarray*}
I^\eps_\sfT(z,r,x;t)\boldsymbol{\Rightarrow} I_\sfT(z,r)
\end{eqnarray*}
    in distribution as $\eps\to 0$ where $I_\sfT(z,r)$ is a non negative random vector with moments given by 
    \begin{eqnarray*}
        \mathbb{E}[I_\sfT(z,r)^p]=\frac{1}{p!}\mathbb{E}[I(z,r)^p]F_p(\sfT)\,
    \end{eqnarray*}
    where $\mathbb{E}[I(z,r)^p]$ is given by Corollary~\ref{coro:intensity_large_corr} and $F_p(\sfT)$ by~\eqref{eqn:F_p(T)}.
\end{theorem}
Now let 
\begin{eqnarray*}
   F_2(\sfT)=1+ F_\sfT,\quad F_\sfT=\frac{1}{\sfT^2}\int\limits_{[0,\sfT]^2}F\Big(\frac{t_1-t_2}{\tau_s}\Big)^2\mathrm{d}t_1\mathrm{d}t_2.
\end{eqnarray*}
We note that $F_\sfT\to1$ as $\sfT\to0$ while $F_\sfT\to0$ as $\sfT\gg \tau_s$. Then we have the following Corollary.
\begin{corollary}[Time averaged scintillation]\label{coro:scint_large_correl}In the diffusive regime,
       \begin{equation}\label{eqn:E_I^2_large_corr}
           \mathbb{E}[I^\eps_\sfT(z,r,x;t)^2]\to\big(\mathbb{E}[I](z,r)^2+\chi(z,r)\big)(1+F_\sfT)\,,
       \end{equation}
    as $\eps\to 0$   with $\mathbb{E}[I](z,r)$ as in~\eqref{eqn:E_I_large_correl_beta=1}, \eqref{eqn:G_large_correl_beta=1} and $\chi(z,r)\le \mathbb{E}[I](z,r)^2$ is given by
              \begin{equation}
       \chi(z,r)=\chi(z,r;\theta)=
           \begin{cases}
               &\mathbb{E}[I](z,r)^2,\quad \beta>1\\
                  &\Big(\frac{12\theta}{z^3}\Big)^d\frac{1}{|\Xi|}\int\limits_{\mathbb{R}^{2d}}e^{\frac{12}{z^3}(r-r')^\tps\Xi^{-1}(r-r')}e^{\frac{3\theta^2}{z^3}\sigma'^\tps\Xi^{-1}\sigma'}\Gamma(r',\sigma')^2\frac{\mathrm{d}r'\mathrm{d}\sigma'}{(2\pi)^d},\quad\beta=1\,.
               \end{cases}
       \end{equation}
              In particular, the scintillation index~\eqref{eqn:scint_time_avg} is asymptotically upper bounded by $3$ with
       \begin{equation}\label{eqn:scint_time_avg_eps}
           \sfS_\sfT^\eps(z,r,x;t)\to \sfS_\sfT(z,r)\,,
       \end{equation}
    as $\eps\to 0$   where
       \begin{eqnarray}\label{eqn:Scint_large_corr}
           \sfS_\sfT(z,r)=\begin{cases}
               1 + 2 F_\sfT,\quad &\beta>1\\
               F_\sfT+\frac{\chi(z,r)}{\mathbb{E}[I](z,r)^2}(1+F_\sfT),\quad & {\beta=1} \\
               F_\sfT, \quad  &\beta=1, \ \theta\to 0\,.
           \end{cases}
       \end{eqnarray}
Moreover when $\beta=1$, 
\begin{eqnarray}\label{eqn:Scint_z_infty}
    \lim_{z\to\infty} \sfS_\sfT(z,r)= F_\sfT+\frac{{\theta^d} \dint_{\mathbb{R}^{2d}}\Gamma(r,\sigma)^2\mathrm{d}r\mathrm{d}\sigma}{\Big(\dint_{\mathbb{R}^{d}}\Gamma(r,0)\mathrm{d}r\Big)^2}(1+F_\sfT)\,.
\end{eqnarray}
\end{corollary}
\begin{theorem}[Stochastic continuity, tightness and convergence of processes]\label{thm:tightness}
    Assume Hypothesis~\ref{hyp:medium} on $R$ along with the assumption that $\langle\xi\rangle^{2}\hat{R}(\xi)\in\sL^1(\mathbb{R}^d)$. Also, assume that the partially coherent source satisfies $\langle\omega\rangle^2\hat{F}(\omega)\in\sL^1(\mathbb{R})$ and
    for every permutation $\pi_n$,
\begin{equation}\label{eqn:J_reg}    
 \int\limits_{\mathbb{R}^{nd}} \prod\limits_{j=1}^n\langle k_j\rangle^{2+\alpha/2}\langle k_{\pi_n(j)}\rangle^{2+\alpha/2}\hat{J}(k_j,k_{\pi_n(j)})\mathrm{d}k_1\ldots\mathrm{d}k_n\le C(n,\alpha)\,
\end{equation}
for some $\alpha>\frac{d-1}{2}$. Then
    \begin{equation}\label{eqn:tightness_crit}
     \sup_{s\in[0,z]}\mathbb{E}|\phi^\eps(s,r,x+h;t+\Delta t)-\phi^\eps(s,r,x;t)|^{2n}\le C(z,n)(|h|^{2\alpha_0 n}+|\Delta t|^{2n}),\quad h\in \sfB(0,1)\subset\mathbb{R}^d, |\Delta t|<1\,
\end{equation}
for $0<\alpha_0\le 1$ with $\alpha_0<1$ in the diffusive regime uniformly in $\eps$. This tightness criterion, along with the convergence of finite dimensional distributions as above shows that for fixed $r$ and $z>0$ each such processes converge in distribution as probability measures on $C^0(\mathbb{R}^d\times\mathbb{R})$. 
\end{theorem}

The proofs of the above theorems and corollaries are presented in Sections \ref{sec:higher_moments} and \ref{sec:limiting_intensity} while the tightness result is presented in an Appendix. These proofs follow from a careful analysis of the limiting moments of the random vector $\Phi^\eps$. The $p+p$th moments of such random vectors do not grow faster than $C^pp!^2$, so that the limiting distributions are characterized by their moments~\textcolor{black}{\cite{billingsley2017probability, durrett2019probability}}. 

Additional smoothness assumptions on the source and medium allow to show the tightness criterion~\eqref{eqn:tightness_crit} and obtain that \textcolor{black}{for $(z,r)$ fixed,} the process $\phi^\eps(z,r)$ in fact converges in distribution as probability measures on $C^0(\mathbb{R}^d\times\mathbb{R})$~\cite{billingsley2017probability, kunita1997stochastic}. \textcolor{black}{The finite dimensional distributions are independent at different values of $r$. This may be verified as in~\cite[Corollary 2.6]{bal2024complex} since it morally corresponds to an infinite distance compared to the effective correlation length $\eta$.} These results generalize those obtained in \cite{bal2024complex} for coherent sources to the setting of partially coherent incident beams.
\section{Interpretation and illustrations of the results}
\label{sec:illustration}
The limiting distribution of wavefields due to partially coherent sources varies as a function of the spatial coherence length and time averaging at the detector. The distinct behaviour in various regimes of overall beam width and coherence lengths as considered here mainly arises from the broadness of these parameters relative to $\frac{1}{\eps}$. In all the cases, the limiting moments in the diffusive regime depend on the medium statistics solely through the Hessian $\Xi$, leading to a slightly simplified model. In particular, the second and fourth moments of the limiting intensity can be described via explicit high dimensional diffusion equations.
   
\paragraph{Broad source with broad correlation length: }
In the case $\beta>1$, the width of the beam is of the order $\frac{1}{\eps^\beta}\gg \frac{1}{\eps}$. The correlation length scales as $\frac{\theta}{\eps^\beta}$ which is still larger than $\frac{1}{\eps}$. In this situation, we should expect the beam to be relatively stable with little mixing between the effects of the random source and the random medium. This is indeed the case, as when $\beta>1$ the $p+p$th moments in Theorems~\ref{thm:large_corr_kinetic} and~\ref{thm:large_corr_diffusive} are clearly separable into a product of terms, one due to the source and another  arising from the medium. Here, they behave like the moments of a product of two independent random variables, $Z_1(r;t)$ and $Z_2(z,x)$, with $Z_1(r;t)=u_0(r;t)$ and $Z_2(z,x)=Z(z,x)+\mathbb{E}[Z_2(z,x)]$. 

In the kinetic regime, $\mathbb{E}[Z_2](z,x)=e^{-\frac{k_0^2R(0)z}{8}}$ and $Z(z,x)$ is a mean zero complex Gaussian with $\mathbb{E}[Z(z,x)Z^\ast(z,y)]=\textcolor{black}{\mathcal{R}_z}(y-x,0)e^{-\frac{k_0^2R(0)z}{4}}$ and $\textcolor{black}{\mathcal{R}_z}$ defined in~\eqref{eqn:cal_R}. 

In the diffusive regime, $\mathbb{E}[Z_2](z,x)=0$ and $Z_2(z,x)$ is simply a mean zero complex Gaussian with covariance $\mathbb{E}[Z_2(z,x)Z_2^\ast(z,y)]=e^{\frac{k_0^2z}{8}(y-x)\Xi (y-x)}$. 

This is also reflected in the distribution of intensity in the diffusive regime, where the limiting intensity~\eqref{eqn:I} behaves like a product of two independent exponential distributions, say $\mathrm{Exp}(1/\Gamma(r,0))$ and $\mathrm{Exp}(1)$. In this regime, the scintillation index is also high, with scintillation approaching 3 as the averaging time of the detector gets smaller, i.e, $F_\sfT$ approaching $1$ for $\sfT\ll \tau_s$.

    \paragraph{Intermediately broad source with intermediate correlation length: }
In the regime $\beta=1$ and $\theta$ of order $O(1)$, both the width of the beam and the coherence are comparable to $\frac{1}{\eps}$. In this case we should expect the beam to fully interact with the fluctuations in the medium. This is reflected in the convolution structure of moments in Theorems~\ref{thm:large_corr_kinetic} and~\ref{thm:large_corr_diffusive}. In the diffusive regime, this convolution appears as a convolution against a diffusion kernel. The limiting scintillation in such an interactive regime is also observed to be greater than one, but still smaller than the previous case as the coefficient of cross correlation $\frac{\chi(z,r)}{\mathbb{E}[I](z,r)^2}$ is strictly less than one in this scaling. 

    \paragraph{Intermediately broad source with small correlation length: }
    When $\beta=1$ and $\theta\ll1$, the overall width of the beam is still comparable to $\frac{1}{\eps}$. However, cross correlations in the beam at an order of $\frac{\theta}{\eps}\ll \frac{1}{\eps}$ do not appear in the limiting distributions. This leads to a considerable reduction in the number of terms cropping up from cross correlations in the limiting moments, reverting back to a complex Gaussian type statistics for the limiting wave field. In particular, the limiting intensity follows an exponential distribution~\eqref{eqn:I}. As seen from~\eqref{eqn:Scint_large_corr}, the reduction in scintillation compared to coherent sources is clearly from an additional time averaging at the detector. The longer is the time averaging, the smaller is the scintillation. This is potentially analogous to the statistical stability achieved through spatial averaging for deterministic sources as well~\cite{bal2024complex}.
    
\begin{remark}
    Even though we have not explicitly considered the case when the overall width of the beam is very large but the correlation length is much smaller than $\frac{1}{\eps}$, we note that the results here will be similar to that of the case $\beta=1, \theta\to 0$, but with second moments as described by the situation when $\beta>1$ in~\eqref{eqn:M_11_kinetic},~\eqref{eqn:M_11_diffusive}. 
\end{remark}

\paragraph{Examples: } As an illustration, we show the limiting behaviour of scintillation in the diffusive regime assuming $\beta=1$.  For simplicity, we take $r=0$ at the center of the beam. From Corollary~\ref{coro:scint_large_correl}, the time averaged scintillation is given by
\begin{equation*}
    S_\sfT(z,0)=F_\sfT+\frac{\chi(z,0)}{\mathbb{E}[I](z,0)^2}(1+F_\sfT)\,.
\end{equation*}
We model the detector averaging function $F$ by
\begin{equation*}
    F(t)=e^{-|t|/\tau_s}\,,
\end{equation*}
so that  
\begin{equation*}
    F_\sfT=\frac{\tau_s}{\sfT}\Big(1-\frac{\tau_s}{2\sfT}(1-e^{-2\sfT/\tau_s})\Big)\,.
\end{equation*}
We consider a spatial covariance of the form
\begin{equation*}
    J_\theta(x,y) = \sff(x)\sff(y)\sfg_\theta(x-y)\,,
\end{equation*}
with $\sff(x)=e^{-\frac{|x|^2}{r_0^2}}$. $\sfg_\theta$ is taken as a Gaussian in the first case and a Bessel function in the second. 
\paragraph{Gaussian correlated beam:} Suppose
\begin{equation*}
\sfg_\theta(x) = e^{-\frac{|x|^2}{2\theta^2r_w^2}}
\end{equation*}
so that
\begin{equation*}
   \frac{\chi(z,0)}{\mathbb{E}[I](z,0)^2}= \Big(\frac{r_s}{r_0}\Big)^d\Big(\frac{|\mathbb{I}-\frac{3\Xi^{-1}r_0^2}{z^3}|}{|\mathbb{I}-\frac{3\Xi^{-1}r_s^2}{z^3}|}\Big)^{1/2}, \quad \frac{1}{r_s^2}=\frac{1}{r_0^2}+\frac{1}{\theta^2r_w^2}\,.
\end{equation*}
In particular, when $\Xi=-\sigma_m^2\mathbb{I}_d$ with $d=2$, we have that
\begin{equation*}
    \frac{\chi(z,0)}{\mathbb{E}[I](z,0)^2} =\frac{1+\frac{\sigma_m^2z^3}{3r_0^2}}{1+\frac{\sigma_m^2z^3}{3r_s^2}}\xrightarrow{z\to\infty}\Big(\frac{r_s}{r_0}\Big)^2=\frac{\theta^2r_w^2}{r_0^2+\theta^2r_w^2}\,.
\end{equation*}
Also, in the limit $r_0\to \infty$, we have
\begin{equation*}
   \frac{\chi(z,0)}{\mathbb{E}[I](z,0)^2}\to \frac{1}{1+\frac{\sigma_m^2z^3}{3\theta^2r_w^2}}\,,
\end{equation*}
which is also in agreement with the results in~\cite{garnier2022scintillation} when the medium is assumed to decorrelate slowly.

\begin{figure}[ht!]
    \centering
    \begin{subfigure}[t]{0.5\textwidth}
    \centering
       \includegraphics[width=8cm]{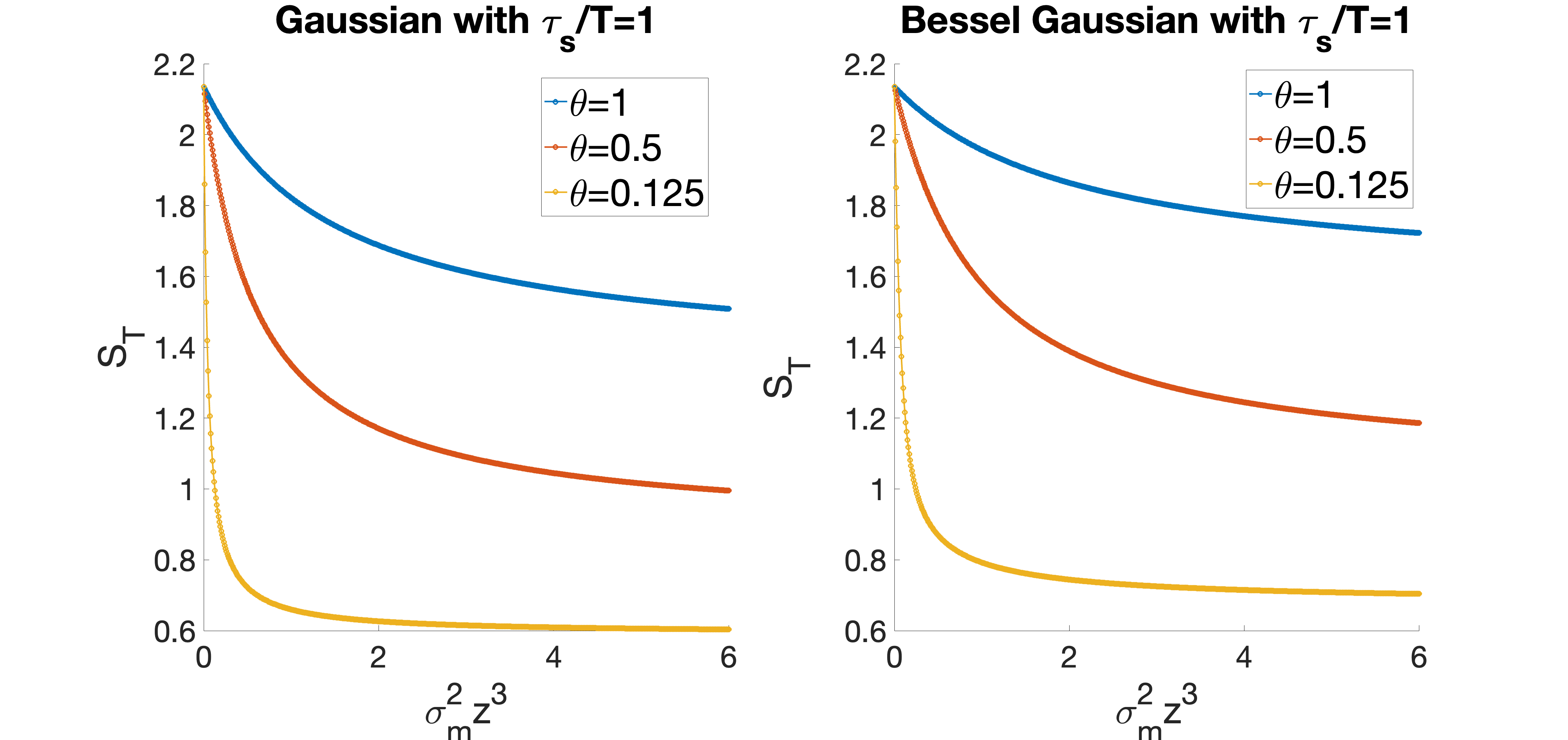}
       \caption{}
    \label{fig:theta_comp}
    \end{subfigure}%
    \begin{subfigure}[t]{0.5\textwidth}
         \centering
       \includegraphics[width=8cm]{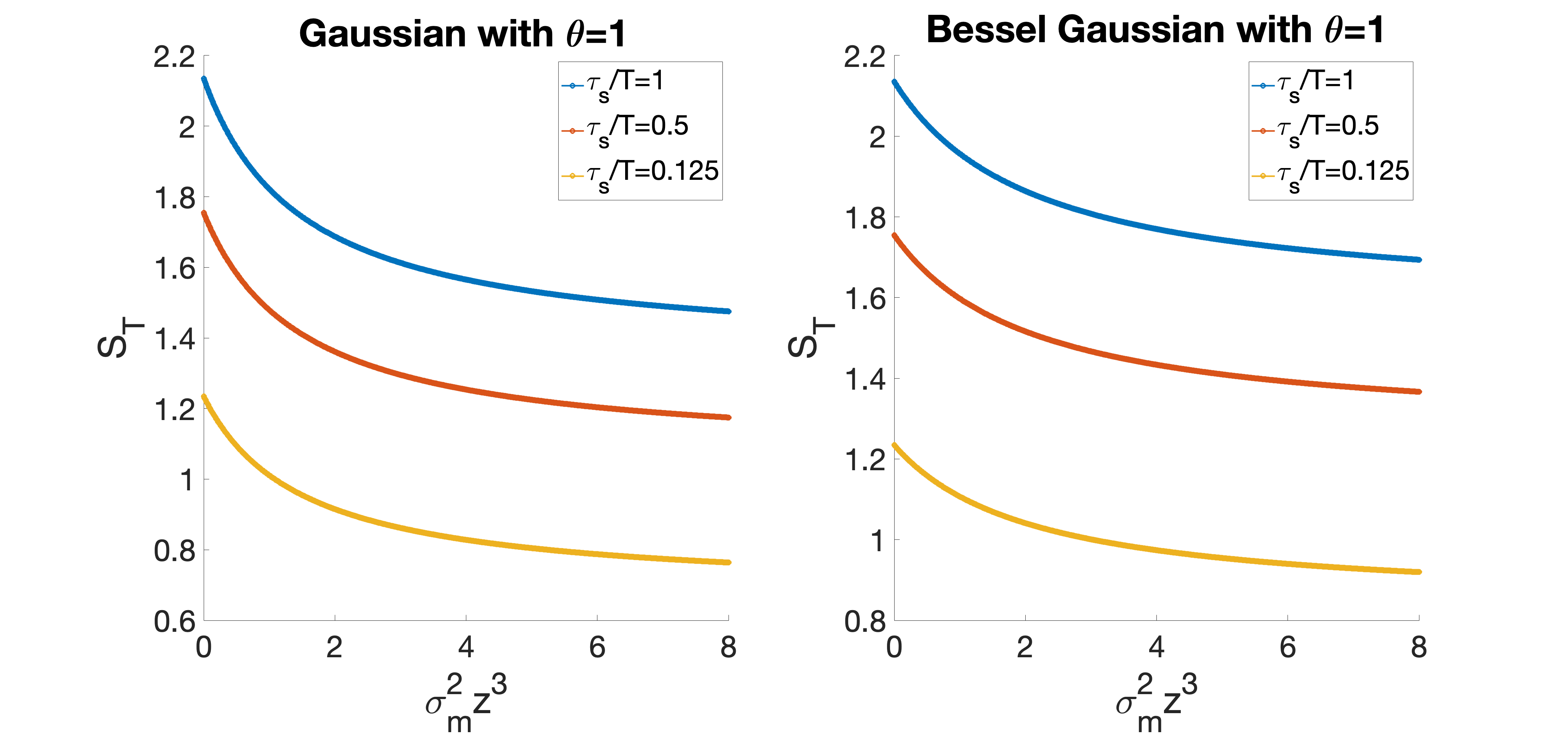}
              \caption{}
    \label{fig:F_T_comp}
    \end{subfigure}
    \caption{Time averaged scintillation $\sfS_\sfT$ as a function of $\sigma_m^2z^3$.  Scintillation saturates to around 1.5, 1 and 0.6 in the first panel and 1.7, 1.2 and 0.7 in the second panel of~\ref{fig:theta_comp}.  In~\ref{fig:F_T_comp}, scintillation saturates to around 1.5, 1.2 and 0.8 in the first panel and 1.7, 1.4 and 1 in the second.} 
        \label{fig:z_plot}
\end{figure}

\begin{figure}[ht!]
    \centering
    \begin{subfigure}[t]{0.5\textwidth}
    \centering
       \includegraphics[width=5cm]{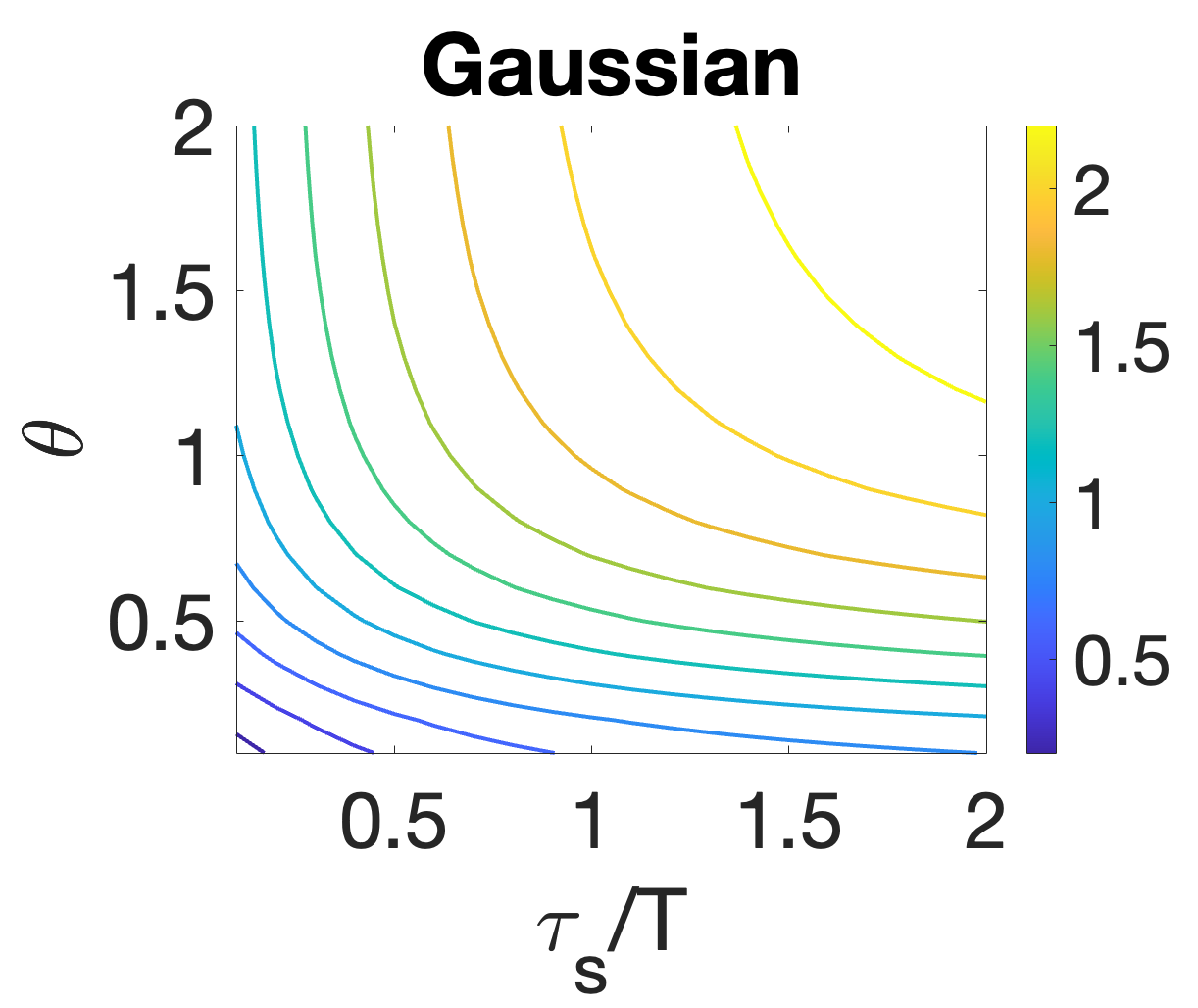}
       \caption{}
    \label{fig:contour_Gaussian}
    \end{subfigure}%
    \begin{subfigure}[t]{0.5\textwidth}
         \centering
       \includegraphics[width=5cm]{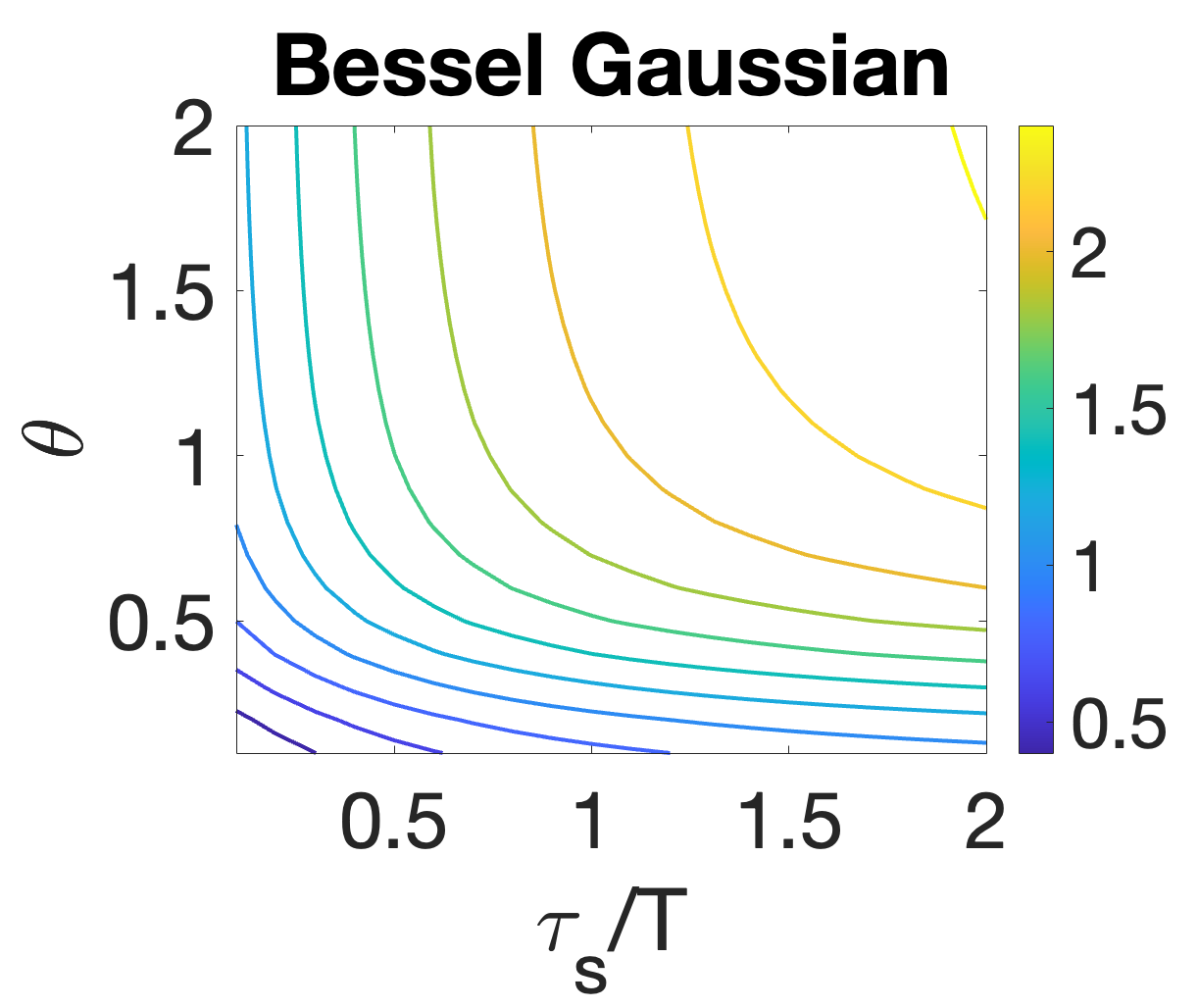}
              \caption{}
    \label{fig:contour_Bessel}
    \end{subfigure}
    \caption{Time averaged scintillation for various $\tau_s/\sfT$ and $\theta$ at a fixed value of $\sigma_m^2z^3=1$.} 
    \label{fig:contour}
\end{figure}
\paragraph{Bessel correlated beam:} Now, suppose
\begin{equation*}
    g_\theta(x)=J_0\Big(\frac{|x|}{\theta r_w}\Big)\,,
\end{equation*}
where $J_0$ denotes the Bessel function of first kind and order 0. Let $\Xi=-\sigma_m^2\mathbb{I}_d$, with $d=2$. In this case both $\chi(z,r)$ and $\mathbb{E}[I](z,r)$ can be evaluated explicitly (see~\cite{gradshteyn2014table}, 6.633). This gives
\begin{equation*}
   \frac{\chi(z,0)}{\mathbb{E}[I](z,0)^2}=I_0\Big(\frac{R(z)^2}{2\theta^2r_w^2}\Big)e^{-\frac{R(z)^2}{2\theta^2r_w^2}},\quad \frac{1}{R(z)^2}=\frac{1}{r_0^2}+\frac{3}{\sigma_m^2z^3}\,,
\end{equation*}
where $I_0$ is the Bessel function of order 0 with imaginary argument. In the limit $z\to \infty$, we have $R(z)^2\to r_0^2$ and
\begin{equation*}
    \lim_{z\to\infty}\frac{\chi(z,0)}{\mathbb{E}[I](z,0)^2}=I_0\Big(\frac{r_0^2}{2\theta^2r_w^2}\Big)e^{-\frac{r_0^2}{2\theta^2r_w^2}}\,.
\end{equation*}

Scintillations computed in these examples also align with the limiting expressions in~\eqref{eqn:Scint_z_infty}. The time averaged scintillation for various configurations of $\tau_s/\sfT$ and $\theta$ are plotted in Fig~\ref{fig:z_plot} and Fig~\ref{fig:contour}. As predicted from the calculations, increasing time averaging $\tau_s/\sfT\to0$ and reducing coherence length $\theta\to0$ both clearly favour a reduction in scintillation.The rest of the paper is devoted to the proof of our main results. 

\section{Two-point statistics of the wave field}
\label{sec:second_moments}
Due to the mean zero property of the source, the first moment of the wave field $\mathbb{E}[u^\eps(z,x;t)]$ is trivially zero. Define the second moment $\mu_{1,1}^\eps(z,x,y;t,t')=\mathbb{E}[u^\eps(z,x;t)u^{\eps *}(z,y;t')]$ with a similar definition for $m^\eps_{1,1}$, the second moment of $\phi^\eps$. Then from It\^o's stochastic calculus~\cite{miyahara1982stochastic}, $\mu_{1,1}^\eps$ solves the PDE
\begin{eqnarray}
    \partial_z\mu_{1,1}^\eps=\frac{i\eta}{2k_0\eps}(\Delta_x-\Delta_y)\mu^\eps_{1,1}+\frac{k_0^2}{4\eta^2}[R(x-y)-R(0)]\mu^\eps_{1,1},\quad\mu^\eps_{1,1}(0,x,y;t,t')=F(t-t')J^\eps(x,y)\,.
\end{eqnarray}
\textcolor{black}{We recall that in order to simplify notation, we systematically replaced $F(\frac t{\tau_s})$ by $F(t)$.} \textcolor{black}{The equations for $\mu^\eps_{1,1}$ and $m^\eps_{1,1}$} can be solved explicitly~\cite{garnier2016fourth, bal2024complex} to give
\begin{equation*}
\begin{aligned}
    m_{1,1}^\eps(z,r,x,y;t,t')&=F(t-t')\int\limits_{\mathbb{R}^{2d}}\hat{\Gamma}^\eps(\zeta,\xi)e^{i\eps^{-\beta}r\cdot\zeta}e^{i\eta\xi\cdot(x-y)}e^{i\eta\zeta\cdot\frac{(x+y)}{2}}e^{-\frac{iz\eta}{k_0\eps}\xi\cdot\zeta}\mQ\big(\eta(y-x),\frac{\eta\zeta}{\eps}\big)\frac{\mathrm{d}\xi\mathrm{d}\zeta}{(2\pi)^{2d}}\,,
\end{aligned}
\end{equation*}
where we have defined:
\begin{eqnarray}\label{eqn:Q}
   \textcolor{black}{\mathcal{Q}_z}(\tau,\tau')=e^{\frac{k_0^2z}{4\eta^2}\int\limits_{0}^1Q\big(\tau+\frac{\tau'sz}{k_0}\big)\mathrm{d}s}, \quad  Q(x)=R(x)-R(0)\,.
\end{eqnarray}
From a change of variables $\xi\to\frac{\eps^\beta}{\theta}\xi, \zeta\to\eps^\beta\zeta$, we have
\begin{equation*}
\begin{aligned}
    m_{1,1}^\eps(z,r,x,y;t,t')&=F(t-t')\!\! \int\limits_{\mathbb{R}^{2d}}\!\! \hat{\Gamma}(\zeta,\xi)e^{ir\cdot\zeta}e^{i\eta\eps^\beta\xi\cdot\frac{(x-y)}{\theta}}e^{i\eta\eps^\beta\zeta\cdot\frac{(x+y)}{2}}e^{-\frac{iz\eta\eps^{2\beta-1}}{k_0\theta}\xi\cdot\zeta}\mQ\big(\eta(y-x),\eta\eps^{\beta-1}\zeta\big)\frac{\mathrm{d}\xi\mathrm{d}\zeta}{(2\pi)^{2d}}.
\end{aligned}
\end{equation*}
Note that in the kinetic regime, $\eta=1$ and we have
\begin{equation}
    \lim_{\eps\to 0}\mQ\big(\eta(y-x),\eta\eps^{\beta-1}\zeta\big)=\begin{cases}
        \mQ(y-x,0),\quad \beta>1\\
        \mQ(y-x,\zeta),\quad\beta=1\,.
    \end{cases}
\end{equation}
Recall that $\lim_{\eps\to 0}m_{1,1}^\eps(z,r,x,y;t,t')=M_{1,1}(z,r,x,y;t,t')$. From an application of the dominated Lebesgue theorem, we have as in~\cite{bal2024complex}
\begin{equation}\label{eqn:M_11_kinetic}
    M_{1,1}(z,r,x,y;t,t')=\begin{cases}
       F(t-t') \Gamma(r,0)\mQ(y-x,0),\quad\beta>1\\
       F(t-t') \int\limits_{\mathbb{R}^{2d}}\Gamma(r',0)\mQ(y-x,\zeta)e^{i\zeta\cdot(r-r')}\frac{\mathrm{d}\zeta\mathrm{d}r'}{(2\pi)^d},\quad\beta=1\,.
    \end{cases}
\end{equation}

Similarly, in the diffusive regime $\eta\ll 1$,
\begin{eqnarray}
   \lim_{\eps\to 0}\mQ\big(\eta(y-x),\eta\eps^{\beta-1}\zeta\big)=
   \begin{cases}
   \exp\Big(\frac{k_0^2z}{8}(y-x)^\tps\Xi(y-x)\Big),\quad \beta>1\\
       \exp\Big(\frac{k_0^2z}{8}\int\limits_0^1\big(y-x+\frac{\zeta sz}{k_0}\big)^\tps\Xi\big(y-x+\frac{\zeta sz}{k_0}\big)\mathrm{d}s\Big),\quad \beta=1\,,
   \end{cases}
\end{eqnarray}
where $\Xi=\nabla^2Q(0)$. This implies that $M_{1,1}(z,r,x,y;t,t')$ is given by:
\begin{eqnarray}\label{eqn:M_11_diffusive}
    \begin{cases}
       F(t-t') \Gamma(r,0)\exp\Big(\frac{k_0^2z}{8}(y-x)^\tps\Xi(y-x)\Big),\quad\beta>1\\
          F(t-t')\int\limits_{\mathbb{R}^{2d}}\Gamma(r',0)e^{i\zeta\cdot(r-r')}\exp\Big(\frac{k_0^2z}{8}\int\limits_0^1\big(y-x+\frac{\zeta sz}{k_0}\big)^\tps\Xi\big(y-x+\frac{\zeta sz}{k_0}\big)\mathrm{d}s\Big)\frac{\mathrm{d}\zeta\mathrm{d}r'}{(2\pi)^{d}},\quad\beta=1\,.
    \end{cases}
\end{eqnarray}
From~\eqref{eqn:M_11_diffusive}, the average limiting intensity $\mathbb{E}[I](z,r)= M_{1,1}(z,r,x,x;t,t)$ in the diffusive regime is given by:
\begin{equation}\label{eqn:E_I}
     \mathbb{E}[I](z,r)=\begin{cases}
         \Gamma(r,0),\quad \beta>1\\
         \int\limits_{\mathbb{R}^{2d}}e^{i\zeta\cdot (r-r')}e^{\frac{z^3}{24}\zeta^\tps\Xi\zeta}{\Gamma}(r',0)\frac{\mathrm{d}\zeta\mathrm{d}r'}{(2\pi)^{d}},\quad\beta=1\,.
     \end{cases}
\end{equation}
In particular when $\beta=1$, $\mathbb{E}[I](z,r)$ follows a diffusion equation given by
\begin{equation}\label{eqn:E_I_large_correl_beta=1}
    \mathbb{E}[I](z,r)=[G(z^3,\cdot)*\Gamma(\cdot,0)](r)\,,
\end{equation}
where $G(t,r)$ is the Green's function to the diffusion equation
\begin{eqnarray}\label{eqn:G_large_correl_beta=1}
    \partial_tG+\frac{1}{24}\nabla_r\cdot(\Xi\nabla_rG)=0,\quad G(0,r)=\delta_0(r)\,.
\end{eqnarray}
Note that in all the cases above, the limiting second moment depends on the source only through $\Gamma(r,0)$, and is therefore independent of $\theta$.

\section{Higher moments of the wave field}
\label{sec:higher_moments}
For $p,q\ge 1$, the $p+q$th moment $\mu_{p,q}^{\eps}(z,X,Y;T)$ of $u^\eps$ is given in \eqref{eq:muepspq} for $(X,Y)=(x_1,\cdots,x_p,y_1,\cdots,y_q)$ and $T=(t_1,\cdots,t_p,t'_1,\cdots,t'_q)$. From  It\^{o}'s stochastic calculus, we obtain~\cite{miyahara1982stochastic, garnier2016fourth}
\begin{equation}\label{eqn:mu_pq_PDE}
\partial_z\mu^\eps_{p,q}=\frac{i\eta}{2k_0\eps}(\sum\limits_{j=1}^p\Delta_{x_j}-\sum\limits_{l=1}^q\Delta_{y_l})\mu^\eps_{p,q}+\frac{k_0^2}{4\eta^2}\mathcal{U}_{p,q}\mu^\eps_{p,q}\,,
\end{equation}
where the generalized potential $\mathcal{U}_{p,q}$ is given by
\begin{equation*}
\mathcal{U}_{p,q}(X,Y)=\sum\limits_{j=1}^p\sum\limits_{l=1}^qR(x_j-y_l)-\sum\limits_{1\le j<j'\le p}R(x_j-x_{j'})-\sum\limits_{1\le l<l'\le q}R(y_l-y_{l'})-\frac{p+q}{2}R(0).
\end{equation*}
The source to~\eqref{eqn:mu_pq_PDE} is given by
\begin{eqnarray}\label{eeqn:mu_pq_source}
    \begin{aligned}
        \mu^\eps_{p,q}(z=0,X,Y;T)=\begin{cases}
            0,\quad p\neq q\\
            \sum\limits_{\pi_p}\prod\limits_{j=1}^pF(t_j-t'_{\pi_p(j)})J^\eps(x_j,y_{\pi_p(j)}),\quad p=q\,.
        \end{cases}
    \end{aligned}
\end{eqnarray}

Following the framework in~\cite{bal2024complex}, the analysis of the moments is primarily done in the Fourier domain in terms of operators acting on $\mathcal{M}_B(\mathbb{R}^{(p+q)d})$. We briefly recall this here. Define the Fourier transformed moments
\begin{equation*}
    \hat{\mu}^\eps_{p,q}(z,v;T)=\int\limits_{\mathbb{R}^{(p+q)d}}\mu^\eps_{p,q}(z,X,Y;T)e^{-i(\sum_{j=1}^p\xi_j\cdot x_j-\sum_{l=1}^q\zeta_l\cdot y_l)}\mathrm{d}x_1\cdots\mathrm{d}x_p\mathrm{d}y_1\cdots\mathrm{d}y_q\,,
\end{equation*}
where $v=(\xi_1,\cdots,\xi_p,\zeta_1,\cdots,\zeta_q)$ denotes the vector of variables dual to $(X,Y)$. Then $\hat{\mu}_{p,q}^\eps$ satisfies
\begin{equation}\label{eq:mupqF}   
\begin{aligned}
        {\partial_z\hat{\mu}^\eps_{p,q}}&=-\frac{i\eta}{2k_0\eps}\Big(\sum_{j=1}^p|\xi_j|^2-\sum_{l=1}^q|\zeta_l|^2\Big)\hat{\mu}^\eps_{p,q}+\frac{k_0^2}{4\eta^2}\int_{\Rm^d}\hat{R}(k)\Big[\sum_{j=1}^p\sum_{l=1}^q\hat{\mu}^\eps_{p,q}(\xi_j-k,\zeta_l-k)\\
    & \ \ \ \ -\sum_{1\le j< j'\le p} \!\!\! \hat{\mu}^\eps_{p,q}(\xi_j-k,\xi_{j'}+k)-\sum_{1\le l< l'\le q}\hat{\mu}^\eps_{p,q}(\zeta_l-k,\zeta_{l'}+k)-\frac{p+q}{2}\hat{\mu}^\eps_{p,q}\Big] \dfrac{\mathrm{d}k}{(2\pi)^d}\,,
    \end{aligned}
\end{equation}
with initial condition
\begin{eqnarray}\label{eqn:mu_pqF_source}
    \begin{aligned}
        \hat{\mu}^\eps_{p,q}(z=0,v;T)=\begin{cases}
            0,\quad p\neq q\\
            \sum\limits_{\pi_p}\prod\limits_{j=1}^pF(t_j-t'_{\pi_p(j)})\hat{J}^\eps(\xi_j,\zeta_{\pi_p(j)}),\quad p=q\,.
        \end{cases}
    \end{aligned}
\end{eqnarray} 
As in~\cite{bal2024complex} we define $\Theta$ as the $d(p+q)\times d(p+q)$ block diagonal matrix consisting of $d\times d$ blocks of $p$ ones followed by $q$ negative ones. $A_{j,l}$ is defined as a $d(p+q)\times d$ matrix of $d\times d$ blocks with blocks $j$ and $p+l$ as the $d\times d$ identity matrix $\mathbb{I}_d$ and the rest vanishing.  Next, we define the phase compensated moments $\psi_{p,q}^\eps(z,v;T)$ such that
\begin{equation}\label{eqn:pi_pq}
    \hat{\mu}^\eps_{p,q}(z,v;T)= \Pi_{p,q}^\eps(z,v)\psi_{p,q}^\eps(z,v;T),\, \qquad \Pi_{p,q}^\eps(z,v):=e^{-\frac{iz\eta}{2k_0\eps}(\sum_{j=1}^p|\xi_j|^2-\sum_{l=1}^q|\zeta_l|^2)}=e^{-\frac{iz\eta v^\tps\Theta v}{2k_0\eps}}\,.
\end{equation}
Under this formulation, $\psi_{p,q}^\eps$ follows the evolution equation
\begin{equation}\label{eqn:psi_evolution}
   \partial_z\psi_{p,q}^\eps= L^\eps_{p,q}\psi_{p,q}^\eps,\qquad \psi_{p,q}^\eps(0,v;T)=\hat{\mu}^\eps_{p,q}(0,v;T)\,,
\end{equation}
where the operator $L^\eps_{p,q}$ is a sum of bounded integral operators on $\mathcal{M}_B(\mathbb{R}^{(p+q)d})$ described in Appendix~\ref{App:operator}.

Throughout, we denote by $\gamma=(\gamma_1,\gamma_2)$ a pairing $(j,l)$ and $\Lambda_\kappa$ denotes a set of $m(\kappa)$ pairings such that if $\gamma,\gamma'\in\Lambda_\kappa$, $\gamma_{1,2}\neq\gamma_{1,2}'$ with $\sum\limits_{\kappa}m(\kappa)=K=\sum\limits_{m=1}^{p\wedge q}\binom{p}{m}\binom{q}{m}m!$. Also, let         
\begin{equation}\label{eqn:cal_R}
            \textcolor{black}{\mathcal{R}_z}(\tau,\tau')=\exp \Big({\frac{k_0^2z}{4\eta^2}\int\limits_{0}^1R\big(\tau+\frac{\tau'sz}{k_0}\big)\mathrm{d}s} \Big).
        \end{equation}
First we verify the following:
\begin{lemma}\label{lemma:U_gamma_operator}
  For any given pairing $\gamma=(j,l)$, let $U^{\eps}_\gamma$ be the solution operator to
  \begin{equation}
\partial_zU^{\eps}_\gamma=L^1_{j,l}U^{\eps}_\gamma,\quad U^{\eps}_\gamma(0)=\mathbb{I}\,,
\end{equation}
where
\begin{equation*}
        {L}^{1}_{j,l}\rho(v)=\frac{k_0^2}{4\eta^2}\int\limits_{\Rm^d}\hat{R}(k)\rho(v-A_{j,l}k)e^{\frac{iz\eta}{k_0\eps}k^\tps A_{j,l}^\tps\Theta v}\frac{\mathrm{d}k}{(2\pi)^d}\,.
\end{equation*}
Then we have that
\begin{equation}\label{eqn:U_0U_gamma_expn}
U^{\eps}_{j,l}\rho(z,v)=\int\limits_{\mathbb{R}^{2d}}e^{-iw\cdot k}\textcolor{black}{\mathcal{R}_z}\big(w,\frac{\eta}{\eps}A_{j,l}^\tps\Theta v\big)\rho(0,v-A_{j,l}k)\frac{\mathrm{d}k\mathrm{d}w}{(2\pi)^d}\,.
\end{equation}
\begin{proof}
For $\rho_0\in\mathcal{M}_B(\mathbb{R}^{(p+q)d})$, let $\rho(z,v)=[U^\eps_\gamma(z)\rho_0](v)$. Now define the coordinate transformation $\xi_j'=\frac{\xi_j+\zeta_l}{2}, \zeta_l'=\xi_j-\zeta_l$. Also let $\rho'(z,\xi_j',\zeta_l')=\rho(z,\xi_j,\zeta_l)$, where we have explicitly indicated only the coordinates that are affected by the transformation. Then $\rho'(z,\xi',\zeta')$ follows the evolution equation
\begin{equation*}
    \begin{aligned}
        \partial_z\rho'=\frac{k_0^2}{4\eta^2}\int\limits_{\mathbb{R}^d}\hat{R}(k)\rho'(z,\xi_j'-k,\zeta_l')e^{\frac{iz\eta}{\eps k_0}k\cdot \zeta'_l}\frac{\mathrm{d}k}{(2\pi)^d},\quad \rho'(0,\xi_j',\zeta_l')=\rho_0\Big(\xi'_j+\frac{\zeta'_l}{2},\xi'_j-\frac{\zeta'_l}{2}\Big)\,.
    \end{aligned}
\end{equation*}
This can be solved explicitly by Fourier transforming with respect to $\xi_j'$. This gives
\begin{equation*}
    \begin{aligned}
        \rho'(z,\xi_j',\zeta_l')&=\int\limits_{\mathbb{R}^{2d}}e^{-iw\cdot(\xi'_j-k)}\rho'(0,k,\zeta'_l)\textcolor{black}{\mathcal{R}_z}\Big(w,\frac{\eta z\zeta'_l}{\eps k_0}\Big)\frac{\mathrm{d}k\mathrm{d}w}{(2\pi)^d}\,.
    \end{aligned}
\end{equation*}
In terms of the original coordinates $\xi_j,\zeta_l$, this translates to
\begin{equation*}
    \begin{aligned}
        \rho(z,\xi_j,\zeta_l)&=\int\limits_{\mathbb{R}^{2d}}e^{-iw\cdot\big(\frac{\xi_j+\zeta_l}{2}-k\big)}\rho_0\Big(k+\frac{\xi_j-\zeta_l}{2},k-\frac{\xi_j-\zeta_l}{2}\Big)\textcolor{black}{\mathcal{R}_z}\Big(w,\frac{\eta z(\xi_j-\zeta_l)}{\eps k_0}\Big)\frac{\mathrm{d}k\mathrm{d}w}{(2\pi)^d}\,.
    \end{aligned}
\end{equation*}
After a change of variables $k-\frac{\xi_j+\zeta_l}{2}\to -k$, we arrive at~\eqref{eqn:U_0U_gamma_expn}.
\end{proof}
\end{lemma}
Let $U^\eps$ denote the solution operator to~\eqref{eqn:psi_evolution}, i.e,
\begin{equation}\label{eqn:sol_operator_evolution}
     \partial_zU^\eps=L^\eps_{p,q}U^\eps,\qquad U^\eps(0)=\mathbb{I}\,.
 \end{equation}
In general, the integral operators in the expansion of $L^\eps_{p,q}$ do not all commute with each other for different values of $z$. However, the rapid phase oscillations in the integrals leads to several simplifications which was leveraged in~\cite{bal2024complex} to decompose $U^\eps$. We summarize this below. Let $U_\eta(z)=e^{-\frac{k_0^2R(0)z}{4\eta^2}}$. Then the following result holds:
\begin{lemma}\label{lemma:U_pq_decomp}
The solution operator $U^\eps$ to the evolution equation~\eqref{eqn:sol_operator_evolution} admits the decomposition
       \begin{equation}\label{eqn:U=N+E}
          U^\eps=N_{p,q}^{\eps}+E_{p,q}^{\eps}\,,
      \end{equation}
where $N^\eps_{p,q}$ is given by
\begin{equation}\label{eqn:N_def}
          N_{p,q}^{\eps}=U_\eta^{\frac{p+q}{2}}  \sum_{\kappa=1}^K  {\prod_{\gamma\in\Lambda_{\kappa}} (U^{\eps}_\gamma-\mathbb{I})} + U_\eta^{\frac{p+q}{2}}
\end{equation}
     and $E_{p,q}^{\eps}$ is a bounded operator from $\mathcal{M}_B(\mathbb{R}^{(p+q)d})$ to itself such that
     \begin{equation}\label{eqn:E_bound}
       \sup_{0\le z'\le z}\|E_{p,q}^{\eps}(z')\|\le c(p,q,z)\eps^{1/3}.
   \end{equation}
   for $c(p,q,z)$ a constant independent of $\eps$ and $\eta$.
   \end{lemma}
The details can be found in \cite[Theorem 4.2]{bal2024complex}. 
Since $T$ is a fixed vector of parameters that do not affect~\eqref{eq:mupqF} and since $\hat{\mu}_{p,q}^\eps(0,v;T)$ is integrable, Lemma~\ref{lemma:U_pq_decomp} allows one to decompose $\hat{\mu}_{p,q}^\eps$ as
\begin{eqnarray}\label{eqn:mu_pqF_decomp}
    \hat{\mu}_{p,q}^\eps(z,v;T)=\Pi_{p,q}^\eps(z,v)[N_{p,q}^\eps(z)+E^\eps_{p,q}(z)]\hat{\mu}_{p,q}^\eps(0,v;T)\,.
\end{eqnarray}
Due to the form of the source~\eqref{eqn:mu_pqF_source}, $\hat{\mu}_{p,q}^\eps$ is equal to zero trivially when $p\neq q$. For $p=q$, mapping this back to the physical domain gives a slight reformulation of \cite[Theorem 4.12]{bal2024complex}. 
\begin{proposition}[Extension to arbitrary integrable sources]\label{prop:N_approx_phy} 
We have
\begin{equation}\label{eqn:mu_phys_decomp_time}
\mu_{p,p}^\eps(z,X,Y;T)=\tilde{N}_{p,p}^{\eps}(z,X,Y;T)+\tilde{E}_{p,p}^{\eps}(z,X,Y;T)\,,
\end{equation}
where
     \begin{equation}\label{eqn:N_phys_time}
        \begin{aligned}
            \tilde{N}_{p,p}^{\eps}(z,X,Y;T)&=U_\eta^p\sum_{\kappa=1}^{K}\int\limits_{\mathbb{R}^{2pd}} e^{iv^\tps\Theta\sfx}
            e^{-\frac{iz\eta}{2k_0\eps}v^\tps\Theta v}\prod\limits_{\gamma\in\Lambda_\kappa}\Big(\textcolor{black}{\mathcal{R}_z}(-A_\gamma^\tps\Theta \sfx,\frac{\eta}{\eps}A_\gamma^\tps\Theta v)-1\Big)\hat{\mu}^\eps_{p,p}(0,v;T)\prod_{j=1}^p\frac{\mathrm{d}\xi_j\mathrm{d}\zeta_j}{(2\pi)^{2d}}\\
            &+U_\eta^p\int\limits_{\mathbb{R}^{2pd}} e^{iv^\tps\Theta\sfx}e^{-\frac{iz\eta}{2k_0\eps}v^\tps\Theta v}\hat{\mu}^\eps_{p,p}(0,v;T)\prod_{j=1}^p\frac{\mathrm{d}\xi_j\mathrm{d}\zeta_j}{(2\pi)^{2d}}\,,
        \end{aligned}
    \end{equation}
   with $\sfx=(X,Y)$ and 
    \begin{equation}\label{eqn:E_phys_time}
 \|\tilde{E}_{p,p}^\eps(z)\|_\infty \leq c(p,z) \eps^{1/3}\, \sup_{0<\eps\leq 1}\|\hat \mu_{p,p}^\eps(0)\|\,,
\end{equation}
for a constant $c(p,z)$ independent of $\eps$. $\mu^\eps_{p,q}=0$ when $p\neq q$.
\begin{proof}
Inverse Fourier transforming~\eqref{eqn:mu_pqF_decomp} gives a similar decomposition 
\begin{equation}
\mu_{p,p}^\eps(z,X,Y;T)=\tilde{N}_{p,p}^{\eps}(z,X,Y;T)+\tilde{E}_{p,p}^{\eps}(z,X,Y;T)\,,
\end{equation}
with
\begin{equation}\label{eqn:N_phys}
        \begin{aligned}
            \tilde{N}_{p,p}^{\eps}(z,X,Y;T)&=U_\eta^p\int\limits_{\mathbb{R}^{2pd}} e^{iv^\tps\Theta\sfx}
            e^{-\frac{iz\eta}{2k_0\eps}v^\tps\Theta v}
            \Big(\sum_{\kappa=1}^{K}\prod\limits_{\gamma\in\Lambda_\kappa}(U_{\gamma}^{\eps}-\mathbb{I})+\mathbb{I}\Big)\hat{\mu}^\eps_{p,p}(0,v;T)\prod_{j=1}^p\frac{\mathrm{d}\xi_j\mathrm{d}\zeta_j}{(2\pi)^{2d}}\,,
        \end{aligned}
    \end{equation}   
where $\sfx=(X,Y)$ and $\tilde{E}^\eps_{p,p}(z,X,Y;T)$ defined in a similar manner. The bound as in~\eqref{eqn:E_bound} translates to a uniform bound in $(X,Y)$ and $T$ which gives~\eqref{eqn:E_phys_time}.
Using Lemma~\ref{lemma:U_gamma_operator}, we have
 \begin{equation*}
        \begin{aligned}
            \tilde{N}_{p,p}^{\eps}(z,X,Y;T)=U_\eta^p\int\limits_{\mathbb{R}^{2pd}} & e^{iv^\tps\Theta\sfx} e^{-\frac{iz\eta}{2k_0\eps}v^\tps\Theta v} \Big(\sum_{\kappa=1}^{K}\int\limits_{\mathbb{R}^{2m(\kappa)d}}e^{-i\sum_{\gamma\in\Lambda_\kappa}w_{\gamma_1}\cdot k_{\gamma_2}}\hat{\mu}^\eps_{p,p}(0,v-\sum_{\gamma\in\Lambda_\kappa}A_\gamma k_{\gamma_2};T)\\
            \times&\prod_{\gamma\in\Lambda_\kappa}\big(\textcolor{black}{\mathcal{R}_z}(w_{\gamma_1},\frac{\eta}{\eps}A_\gamma^\tps\Theta v)-1\big)\frac{\mathrm{d}w_{\gamma_1}\mathrm{d}k_{\gamma_2}}{(2\pi)^d}+\hat{\mu}^\eps_{p,p}(0,v;T)\Big)\prod_{j=1}^p\frac{\mathrm{d}\xi_j\mathrm{d}\zeta_j}{(2\pi)^{2d}}\,.
        \end{aligned}
    \end{equation*}
From a change of variables $v-\sum_{\gamma\in\Lambda_\kappa}A_\gamma k_{\gamma_2}\to v$ for a fixed $\kappa$, we have
 \begin{equation*}
        \begin{aligned}
            \tilde{N}_{p,p}^{\eps}(z,X,Y;T)=& U_\eta^p\sum_{\kappa=1}^{K}\int\limits_{\mathbb{R}^{2pd}}\int\limits_{\mathbb{R}^{2m(\kappa)d}} e^{iv^\tps\Theta\sfx}e^{i\sum_{\gamma\in\Lambda_\kappa}k_{\gamma_2}^\tps A_\gamma^\tps\Theta\sfx}
            e^{-\frac{iz\eta}{2k_0\eps}v^\tps\Theta v}e^{-\frac{iz\eta}{k_0\eps}\sum_{\gamma\in\Lambda_\kappa}k_{\gamma_2}^\tps A_\gamma^\tps \Theta v}e^{-i\sum_{\gamma\in\Lambda_\kappa}w_{\gamma_1}\cdot k_{\gamma_2}}\\
            &\qquad \times\prod\limits_{\gamma\in\Lambda_\kappa}\Big(\textcolor{black}{\mathcal{R}_z}(w_{\gamma_1},\frac{\eta}{\eps}A_\gamma^\tps\Theta v)-1\Big)\hat{\mu}^\eps_{p,p}(0,v;T)\prod_{\gamma\in\Lambda_\kappa}\frac{\mathrm{d}w_{\gamma_1}\mathrm{d}k_{\gamma_2}}{(2\pi)^d}\prod_{j=1}^p\frac{\mathrm{d}\xi_j\mathrm{d}\zeta_j}{(2\pi)^{2d}}\\
            &+U_\eta^p\int\limits_{\mathbb{R}^{2pd}} e^{iv^\tps\Theta\sfx}e^{-\frac{iz\eta}{2k_0\eps}v^\tps\Theta v}\hat{\mu}^\eps_{p,p}(0,v;T)\prod_{j=1}^p\frac{\mathrm{d}\xi_j\mathrm{d}\zeta_j}{(2\pi)^{2d}}\,.
        \end{aligned}
    \end{equation*}
Here we have used the fact that $A_\gamma^\top\Theta A_{\gamma'}=0$ for $\gamma,\gamma'\in\Lambda_\kappa$ both when $\gamma=\gamma'$ and $\gamma\neq\gamma'$. Note that $k_{\gamma_2}$ appears only in a complex phase of the form
\begin{equation*}
    k_{\gamma_2}\cdot(A_\gamma^\tps\Theta\sfx -\frac{z\eta}{k_0\eps} A_\gamma^\tps\Theta v-w_{\gamma_1})\,,
\end{equation*}
which upon integration gives a factor of $(2\pi)^d\delta_0(A_\gamma^\tps\Theta\sfx -\frac{z\eta}{k_0\eps} A_\gamma^\tps\Theta v-w_{\gamma_1})$. Repeating this for all $\gamma\in\Lambda_\kappa$ and using 
\begin{equation*}
        \int\limits_0^1R\Big(A_\gamma^\tps\Theta \sfx-\frac{z\eta}{k_0\eps}A_\gamma^\tps\Theta v+\frac{sz\eta}{k_0\eps}A_\gamma^\tps\Theta v\Big)\mathrm{d}s=\int\limits_0^1R\Big(-A_\gamma^\tps\Theta \sfx+\frac{sz\eta}{k_0\eps}A_\gamma^\tps\Theta v\Big)\mathrm{d}s\,
    \end{equation*}
gives~\eqref{eqn:N_phys_time}. 
\end{proof}
    \end{proposition}

\begin{remark}
    Observing that $K=\sum\limits_{m=1}^{p}\binom{p}{m}\binom{p}{m}m!\le p!\sum\limits_{m=1}^{p}\binom{p}{m}\le p!2^p$,
\begin{equation}\label{eqn:N_bound}
    \|N^\eps_{p,p}\|\le 1+2^{2p}p!\le C2^{2p}p!\,.
\end{equation} 
Due to the complex Gaussian nature of the source, $\sup_{T}\|\hat{\mu}^\eps_{p,p}(0,\cdot;T)\|\le C^pp!$ as well. This gives for the $p+p$th moment
\begin{equation*}
    \|N^\eps_{p,p}(z)\hat{\mu}^\eps_{p,p}(0)\|\le C^p(p!)^2\le C^p{p}^{2p}=C^pe^{2p\log(2p)}.
\end{equation*}
We obtain the same bound in the physical domain for $\|\mu^\eps_{p,p}(z)\|_\infty$ uniformly in $(X,Y)$ and $T$, which means we can uniquely characterize limiting distributions through their moments, provided the moments converge~\textcolor{black}{\cite{billingsley2017probability, durrett2019probability}}.
\end{remark}
Note that ${N}_{p,q}^{\eps}$ is computed as products of second moments. In particular, for deterministic sources,
\begin{equation}\label{eqn:coh_source}
  \hat{\mu}_{p,q}^\eps(0,v;T)=  \hat{\mu}^\eps_{p,q}(0,v)=\prod\limits_{j=1}^p\hat{u}_0^\eps(\xi_j)\prod\limits_{l=1}^q{\hat{u}_0^\eps}^\ast(\zeta_l)\,,
\end{equation}
and $\tilde{N}^\eps$ can be written as  
\begin{equation*}
  \begin{aligned}
      \mu_{p,q}^{\eps }(z,X,Y)&=\mathscr{F}(\mu^\eps_{1,0}(z,x_1),\ldots,\mu^\eps_{1,0}(z,x_p),\mu^\eps_{0,1}(z,y_1),\ldots,\mu^\eps_{0,1}(z,y_q),\mu^\eps_{1,1}(z,x_1,y_1),\ldots,\mu^\eps_{1,1}(z,x_p,y_q))\\&
      +\tilde{E}^{\eps}(z,X,Y)\,,
  \end{aligned}
      \end{equation*}
  where $\mu_{1,0}^\eps(z,x)=\mu_{0,1}^{\eps *}(z,x)$ and $\mu_{1,1}^\eps(z,x,y)$ solve the first and second moment equations
  \begin{equation}\label{eq:mu1}
    \partial_z\mu_{1,0}^\eps=\frac{i\eta}{2k_0\eps}\Delta_x\mu_{1,0}^\eps-\frac{k_0^2}{8\eta^2}R(0)\mu_{1,0}^\eps,\quad \mu_{1,0}^\eps(0,x)=u_0^\eps(x)\,,
\end{equation}
and
\begin{equation}\label{eqn:mu_2}
    \partial_z\mu_{1,1}^\eps=\frac{i\eta}{2k_0\eps}(\Delta_x-\Delta_y)\mu_{1,1}^\eps+\frac{k_0^2}{4\eta^2}[R(x-y)-R(0)]\mu_{1,1}^\eps,\quad \mu_{1,1}^\eps(0,x,y)=u_0^\eps(x){u_0^\eps}^\ast(y)\,.
\end{equation}
  $\mathscr{F}$ is a bounded and continuous functional of its $p+q+pq$ arguments given explicitly  by
 \begin{equation}\label{eqn:F_def}
 \begin{aligned}     \mathscr{F}(h_1,\cdots,h_p,h'_1,\cdots,h'_q,g_{1,1},\cdots,g_{p,q})&=\sum_{\kappa=1}^{K}\prod_{\gamma\in\Lambda_{\kappa}}[g_{\gamma}-h_{\gamma_1}h'_{\gamma_2}]\prod_{j\notin\Lambda_{\kappa,1}}h_{j}\prod_{l\notin\Lambda_{\kappa,2}}h'_{l} +\prod_{j=1}^ph_j\prod_{l=1}^qh'_l\,.
 \end{aligned}
     \end{equation}
Here $\Lambda_{\kappa,1}$ denotes the set of all indices in the first position in the pairs in $\Lambda_{\kappa}$ and $\Lambda_{\kappa,2}$ denotes the set of all indices in the second position. \textcolor{black}{We note that the functional $\mathcal{F}$ is a relationship between raw moments. This is consistent with the Gaussian summation rule that centred higher order even moments can be written in terms of products of centred second moments.}

When the source does not admit the decomposition~\eqref{eqn:coh_source}, $\tilde{N}^\eps(z,X,Y)$ is still described using first and second moment type equations, however not with initial conditions as in~\eqref{eq:mu1} and~\eqref{eqn:mu_2}. The correlation length scale of the source significantly affects the limiting behaviour of these moments. 

Performing a change of variables $\eps^{-\beta}v\to v$ in~\eqref{eqn:N_phys_time} and noting that $A_\gamma^\tps\Theta \sfx=x_{\gamma_1}-y_{\gamma_2}$, we have the following:
\begin{proposition}\label{prop:m_pp}
The moments of the random vector $\phi^\eps$ are given by
 \begin{equation}\label{eqn:m_pp_alpha=beta}
        \begin{aligned}
            m_{p,p}^{\eps}(z,r,X,Y;T)&=U_\eta^p\sum_{\kappa=1}^{K}\int\limits_{\mathbb{R}^{2pd}} e^{i\sum_{j=1}^p(\xi_j-\zeta_j)\cdot r}e^{i\eps^\beta \eta v^\tps\Theta\sfx}
            e^{-\frac{iz\eta\eps^{2\beta-1}}{2k_0}v^\tps\Theta v}\\
            &\times\prod\limits_{\gamma\in\Lambda_\kappa}\Big(\textcolor{black}{\mathcal{R}_z}\big(\eta(y_{\gamma_2}-x_{\gamma_1}),\eta\eps^{\beta-1}(\xi_{\gamma_1}-\zeta_{\gamma_2})\big)-1\Big)\hat{\Sp}(v;T)\prod_{j=1}^p\frac{\mathrm{d}\xi_j\mathrm{d}\zeta_j}{(2\pi)^{2d}}\\
            &+U_\eta^p\int\limits_{\mathbb{R}^{2pd}}e^{i\sum_{j=1}^p(\xi_j-\zeta_j)\cdot r}e^{i\eps^\beta\eta v^\tps\Theta\sfx}e^{-\frac{iz\eta\eps^{2\beta-1}}{2k_0}v^\tps\Theta v}\hat{\Sp}(v;T)\prod_{j=1}^p\frac{\mathrm{d}\xi_j\mathrm{d}\zeta_j}{(2\pi)^{2d}}+O(\eps^{\frac{1}{3}})\,,
        \end{aligned}
    \end{equation}
    where $\Sp(X,Y;T)$ is defined in \eqref{eqn:Sp}. Also, $m^\eps_{p,q}=0$ when $p\neq q$.
\end{proposition}
 We also have the following corollary:
\begin{corollary}\label{coro:Phi_mom}
    Define the moments of the random vector $\Phi^\eps(z,r,X;T)$ by
    \begin{equation*}
        \mathbf{M}_{p,q}^\eps(z,r,X;T)=\mathbb{E}[\underbrace{\Phi^\eps\otimes\cdots\otimes\Phi^\eps}_{p \text{ terms}}\otimes\underbrace{{\Phi^{\eps\ast}}\otimes\cdots\otimes{\Phi^{\eps\ast}}}_{q \text{ terms}}](z,r,X;T)\,.
    \end{equation*}
    Then the elements of $\mathbf{M}_{p,q}^\eps$ are given by 
    \begin{equation*}
        \mathbb{E}[\prod\limits_{j=1}^p\phi^\eps_{s_j}\prod\limits_{l=1}^q\phi^{\eps *}_{s'_l}]=m^\eps_{p,q}\big(z,r,(x_{s_j},x_{s'_l})_{j,l});(t_{s_j},t_{s'_l})_{j,l})\big)\,,
    \end{equation*}
    where $m_{p,q}^\eps$ is defined in Proposition~\ref{prop:m_pp} and $s_j$ and $s'_l$ are integers drawn from $\{1,\cdots, N\}$ with replacement. 
\end{corollary}

\subsection{Limiting moments in the kinetic regime}\label{subsec:kinetic_mom}
When $\eta(\eps)=1$, we have that
\begin{equation*}
    \begin{aligned}
    \lim_{\eps\to 0}    \textcolor{black}{\mathcal{R}_z}(\eta(y_{\gamma_2}-x_{\gamma_1}),\eta\eps^{\beta-1}(\xi_{\gamma_1}-\zeta_{\gamma_2}))=\begin{cases}
            \textcolor{black}{\mathcal{R}_z}(y_{\gamma_2}-x_{\gamma_1},0),\quad\beta>1\\
            \textcolor{black}{\mathcal{R}_z}(y_{\gamma_2}-x_{\gamma_1}, \xi_{\gamma_1}-\zeta_{\gamma_2}),\quad\beta=1\,.
        \end{cases}
    \end{aligned}
\end{equation*}
In the case $\beta>1$, from Proposition~\ref{prop:m_pp} and the Lebesgue dominated convergence theorem, we have as $\eps\to0$, 
\begin{equation}
\begin{aligned}
   M_{p,p}(z,r,X,Y;T)&= \Sp(r,\ldots,r;T)e^{-\frac{pk_0^2R(0)z}{4}}\big(1+\sum\limits_{\kappa=1}^K\prod\limits_{\gamma\in\Lambda_\kappa}[\textcolor{black}{\mathcal{R}_z}(y_{\gamma_2}-x_{\gamma_1},0)-1]\big)\\
   &=\Sp(r,\ldots,r;T) \mathscr{F}\Big(M^\infty_{1,0}(z,r,x_j)_{j},M^\infty_{0,1}(z,r,y_l)_{l},M^\infty_{1,1}(z,r,x_j,y_l)_{j,l}\Big)\,,
\end{aligned}
     \end{equation}
where $M_{1,0}^\infty(z,r,x)=M_{0,1}^\infty(z,r,x)=e^{-\frac{k_0^2R(0)z}{8}}$ and $M_{1,1}^\infty$ is given by\begin{equation}\label{eqn:M_11_inf_kinetic}
 M_{1,1}^\infty(z,r,x,y)=\textcolor{black}{\mathcal{Q}_z}(y-x,0)\,.
\end{equation}
for $\mQ$ as in~\eqref{eqn:Q}. 

Similarly when $\beta=1$,
\begin{equation*}
\begin{aligned}
 &M_{p,p}(z,r,X,Y;T)=\Sp(r,\ldots,r;T)e^{-\frac{pk_0^2R(0)z}{4}} \\
       &+e^{-\frac{pk_0^2R(0)z}{4}}\sum\limits_{\kappa=1}^K\int\limits_{\mathbb{R}^{2pd}}e^{i\sum_{j=1}^p(\xi_j-\zeta_j)\cdot r}\prod\limits_{\gamma\in\Lambda_\kappa}[\textcolor{black}{\mathcal{R}_z}(y_{\gamma_2}-x_{\gamma_1},\xi_{\gamma_1}-\zeta_{\gamma_2})-1]\hat{\Sp}(v;T)\prod_{j=1}^p\frac{\mathrm{d}\xi_j\mathrm{d}\zeta_j}{(2\pi)^{2d}}\,. 
\end{aligned}
\end{equation*}
This can be recast as
\begin{equation}\label{eqn:M_pp_beta_1_refo}
\begin{aligned}
& M_{p,p}(z,r,X,Y;T)=\Sp(r,\cdots,r;T)e^{-\frac{pk_0^2R(0)z}{4}} +e^{-\frac{pk_0^2R(0)z}{4}}\sum\limits_{\kappa=1}^K\int\limits_{\mathbb{R}^{2pd}}\Sp(X',Y';T)\prod\limits_{j=1}^p\mathrm{d}x'_j\mathrm{d}y'_j\\
 &\times\int\limits_{\mathbb{R}^{pd}}\prod\limits_{\gamma\in\Lambda_\kappa}\delta_0(x'_{\gamma_1}-y'_{\gamma_2})[\textcolor{black}{\mathcal{R}_z}(y_{\gamma_2}-x_{\gamma_1},\zeta_{\gamma_2})-1]e^{i\zeta_{\gamma_2}\cdot(r-y'_{\gamma_2})}\prod\limits_{j\notin\Lambda_{\kappa,1}}\delta_0(r-x'_j)\prod\limits_{l\notin\Lambda_{\kappa,2}}\delta_0(r-y'_l)\prod_{j=1}^p\frac{\mathrm{d}\zeta_j}{(2\pi)^{d}}\,. 
\end{aligned}
\end{equation}
Define $M_{p,p}^\infty$ as
\begin{equation}\label{eqn:M_pp_inf_kinetic}
\begin{aligned}
   &M_{p,p}^\infty(z,X',Y',X,Y)=\prod_{j=1}^p\delta_0(x'_j)e^{-\frac{k_0^2R(0)z}{8}}\prod_{l=1}^p\delta_0(y'_l)e^{-\frac{k_0^2R(0)z}{8}}\\
    &+\sum\limits_{\kappa=1}^K\prod\limits_{\gamma\in\Lambda_\kappa} \overline{M}^\infty_{1,1}(z,y'_{\gamma_2},x_{\gamma_1},y_{\gamma_2})\delta_0(x'_{\gamma_1}-y'_{\gamma_2})\prod\limits_{j\notin\Lambda_{\kappa,1}}\delta_0(x'_{j})e^{-\frac{k_0^2R(0)z}{8}}\prod\limits_{l\notin\Lambda_{\kappa,2}}\delta_0(y'_{l})e^{-\frac{k_0^2R(0)z}{8}}\,,  
\end{aligned}
   \end{equation}
where $\overline{M}^\infty_{1,1}$ is given by
\begin{eqnarray}
    \overline{M}^\infty_{1,1}(z,r,x,y)=e^{-\frac{k_0^2R(0)z}{4}}\int\limits_{\mathbb{R}^d}[\textcolor{black}{\mathcal{R}_z}(y-x,\zeta)-1]e^{i\zeta\cdot r}\frac{\mathrm{d}\zeta}{(2\pi)^d}\,.
\end{eqnarray}
This allows us to write~\eqref{eqn:M_pp_beta_1_refo} as the convolution
\begin{eqnarray}
    \begin{aligned}
        M_{p,p}(z,r,X,Y;T)&= \Big(M_{p,p}^\infty(z,\cdot,X,Y)\ast\Sp(\cdot;T)\Big)(r,\ldots,r)\,.
    \end{aligned}
\end{eqnarray}
 In the case $\beta=1$ and $\theta\ll 1$, we still have
\begin{equation*}
\begin{aligned}
  M^\theta_{p,p}(z,r,X,Y;T)&:= \lim_{\eps\to 0} m^\eps_{p,p}(z,r,X,Y;T)= e^{-\frac{pk_0^2R(0)z}{4}}\Big( \Sp(r,\ldots,r;T)\\
   &+\sum\limits_{\kappa=1}^K\int\limits_{\mathbb{R}^{2pd}}e^{i\sum_{j=1}^p(\xi_j-\zeta_j)\cdot r}\prod\limits_{\gamma\in\Lambda_\kappa}[\textcolor{black}{\mathcal{R}_z}(y_{\gamma_2}-x_{\gamma_1},\xi_{\gamma_1}-\zeta_{\gamma_2})-1]\hat{\Sp}(v;T)\prod\limits_{j=1}^p\frac{\mathrm{d}\xi_j\mathrm{d}\zeta_j}{(2\pi)^{2d}}\,.
\end{aligned}
   \end{equation*}
For a fixed set $\Lambda_\kappa$ in the sum above, we have, using the shorthand notation $\mathrm{d}\tilde\Xi:=\prod\limits_{j=1}^p\frac{\mathrm{d}\xi_j\mathrm{d}\zeta_j}{(2\pi)^{2d}}$ that
\begin{equation*}
    \begin{aligned}
   &I^\theta_0 :=\int\limits_{\mathbb{R}^{2pd}}e^{i\sum_{j=1}^p(\xi_j-\zeta_j)\cdot r}\prod\limits_{\gamma\in\Lambda_\kappa}[\textcolor{black}{\mathcal{R}_z}(y_{\gamma_2}-x_{\gamma_1},\xi_{\gamma_1}-\zeta_{\gamma_2})-1]\hat{\Sp}(v;T) 
\mathrm{d}\tilde\Xi\\ 
     &=   \theta^{pd}\int\limits_{\mathbb{R}^{2pd}}e^{i\sum_{j=1}^p(\xi_j-\zeta_j)\cdot r}\sum\limits_{\pi_p}F(t_j-t'_{\pi_p(j)})\hat{\Gamma}\Big(\xi_j-\zeta_{\pi_p(j)},\frac{\theta(\xi_j+\zeta_{\pi_p(j)})}{2}\Big)\prod\limits_{\gamma\in\Lambda_\kappa}[\textcolor{black}{\mathcal{R}_z}(y_{\gamma_2}-x_{\gamma_1},\xi_{\gamma_1}-\zeta_{\gamma_2})-1] \mathrm{d}\tilde\Xi. 
    \end{aligned}
\end{equation*}
For each $\pi_p$, from a change of variables $\Big(\xi_j-\zeta_{\pi_p(j)},\frac{\theta(\xi_j+\zeta_{\pi_p(j)})}{2}\Big)\to(\zeta_{\pi_p(j)},\xi_j), j=1,\cdots,p$ we have
\begin{equation*}
      I^\theta_0=  \int\limits_{\mathbb{R}^{2pd}}e^{i\sum_{j=1}^p\zeta_j\cdot r}\sum\limits_{\pi_p}\prod\limits_{j=1}^pF(t_j-t'_{\pi_p(j)})\hat{\Gamma}(\zeta_{\pi_p}(j),\xi_j)\prod\limits_{\gamma\in\Lambda_\kappa}[\textcolor{black}{\mathcal{R}_z}\Big(y_{\gamma_2}-x_{\gamma_1},\frac{\xi_{\gamma_1}-\xi_{\pi_p^{-1}(\gamma_2)}}{\theta}+\frac{\zeta_{\pi_p(\gamma_1)}+\zeta_{\gamma_2}}{2}\Big)-1]
      \mathrm{d}\tilde\Xi. 
   \end{equation*}
   Note that if $\gamma_1\neq\pi_p^{-1}(\gamma_2)$, i.e, at least one pairing in $\pi_p$ is not the same as the pairings in $\Lambda_\kappa$ for the same variables, we have that
   \begin{equation*}
       \lim_{\theta\to 0}\textcolor{black}{\mathcal{R}_z}\Big(y_{\gamma_2}-x_{\gamma_1},\frac{\xi_{\gamma_1}-\xi_{\pi_p^{-1}(\gamma_2)}}{\theta}+\frac{\zeta_{\pi_p(\gamma_1)}+\zeta_{\gamma_2}}{2}\Big)=1\,.
   \end{equation*}   
   This means that in the limit $\theta\to 0$, only those pairings in $\Lambda_\kappa$ that are the same as the pairings in $\pi_p$ contribute and we have as $\theta\to 0$
   \begin{eqnarray}
       \begin{aligned}
          & M_{p,p}(z,r,X,Y;T)=e^{-\frac{pk_0^2R(0)z}{4}}\Big( \Sp(r,\ldots,r;T)\\
   &+\sum\limits_{\kappa=1}^K\prod\limits_{\gamma\in\Lambda_\kappa}\Big(M_{1,1}(z,r,x_{\gamma_1},y_{\gamma_2};t_{\gamma_1},t'_{\gamma_2})-M_{0,0}(z,r;t_{\gamma_1},t'_{\gamma_2})\Big)\sum\limits_{\pi_{\bar{m}(\kappa)}}\prod\limits_{j=1}^{\bar{m}(\kappa)}M_{0,0}(z,r;t_{s_j},t'_{s'_{\pi_{\bar{m}(\kappa)}}(j)})\,,
       \end{aligned}
   \end{eqnarray}
   where $M_{1,1}$ is defined in~\eqref{eqn:M_11_kinetic} for $\beta=1$ and $M_{0,0}$ is given by
   \begin{eqnarray}\label{eqn:M_00_kinetic}
       M_{0,0}(z,r;t_{\gamma_1},t'_{\gamma_2})=F(t_{\gamma_1}-t'_{\gamma_2})\Gamma(r,0)e^{-\frac{k_0^2R(0)z}{4}}\,.
   \end{eqnarray}
Here, $\bar{m}(\kappa)=p-m(\kappa)$, $\{s_j\}_j=\{j\notin\Lambda_{\kappa,1}\}$ and $\{s'_l\}_l=\{l\notin\Lambda_{\kappa,2}\}$. This can be written as 
\begin{eqnarray}
    M_{p,p}(z,r,X,Y;T)=\mathscr{G}\big(M_{0,0}(z,r;t_j,t'_l)_{j,l},M_{1,1}(z,r,x_j,y_l;t_j,t'_l)_{j,l}\big)
\end{eqnarray}
where $\mathscr{G}$ is a functional of $2p^2$ arguments
   \begin{equation}\label{eqn:G_def}
    \begin{aligned}
        \mathscr{G}(h_{1,1},\cdots,h_{p,p},g_{1,1},\cdots,g_{p,p})&=\sum\limits_{\kappa=1}^K\prod\limits_{\gamma\in\Lambda_\kappa}(g_{\gamma_1,\gamma_2}-h_{\gamma_1,\gamma_2})\sum\limits_{\pi_{\bar{m}(\kappa)}}\prod\limits_{j=1}^{\bar{m}(\kappa)}h_{{s_j},{s'_{\pi_{\bar{m}(\kappa)}(j)}}}+\sum\limits_{\pi_p}\prod\limits_{j=1}^ph_{j,\pi_p(j)}\,.
    \end{aligned}
\end{equation}  
This completes the convergence analysis of the moments $m^\eps_{p,q}$ to their limits \eqref{eq:Mppkinetic} in Theorem \ref{thm:large_corr_kinetic}.
\subsection{Limiting moments in the diffusive regime}\label{subsec:diff_mom}
When $\eta\ll 1$, we have that
\begin{equation*}
\begin{aligned}
   &\lim_{\eps\to 0} U_\eta^p\prod\limits_{\gamma\in\Lambda_\kappa}\Big(\textcolor{black}{\mathcal{R}_z}\big(\eta(y_{\gamma_2}-x_{\gamma_1}),\eta\eps^{\beta-1}(\xi_{\gamma_1}-\zeta_{\gamma_2})\big)-1\Big)\\
   &=\lim_{\eps\to 0}e^{-\frac{pk_0^2zR(0)}{4\eta^2}}\prod\limits_{\gamma\in\Lambda_\kappa}\Big(e^{\frac{k_0^2z}{4\eta^2}\int\limits_{0}^1R\big(\eta(y_{\gamma_2}-x_{\gamma_1})+\frac{\eta\eps^{\beta-1}(\xi_{\gamma_1}-\zeta_{\gamma_2})sz}{k_0}\big)\mathrm{d}s}-1\Big)=0\,,   
\end{aligned}
 \end{equation*}
when $m(\kappa)<p$. So in~\eqref{eqn:m_pp_alpha=beta} only the product with $m(\kappa)=p$ contributes to the sum in the limit $\eps\to 0$. Note that for $\beta>1$ we have
\begin{equation*}
    \begin{aligned}
       \lim_{\eps\to 0}\frac{k_0^2z}{4\eta^2}\int\limits_{0}^1[R\big(\eta(y_{\gamma_2}-x_{\gamma_1})+\frac{\eta\eps^{\beta-1}(\xi_{\gamma_1}-\zeta_{\gamma_2})sz}{k_0}\big)-R(0)]\mathrm{d}s=
           \frac{k_0^2z}{8}(y_{\gamma_2}-x_{\gamma_1})^\tps\Xi(y_{\gamma_2}-x_{\gamma_1})\,.
        \end{aligned}
\end{equation*}
As in the kinetic regime, this gives
\begin{eqnarray}
    M_{p,p}(z,r,X,Y;T)=\Sp(r,\ldots,r;T)\sum\limits_{\pi_p}\prod\limits_{j=1}^pM_{1,1}^\infty(z,r,x_j,y_{\pi_p(j)}),\quad\beta>1\,,
\end{eqnarray}
where $M_{1,1}^\infty$ is given by~
\begin{equation}\label{eqn:M_11_inf_diff}
 M_{1,1}^\infty(z,r,x,y)=\exp\Big(\frac{k_0^2z}{8}(y-x)^\tps\Xi(y-x)\Big)\,.
\end{equation}
Similarly, when $\beta=1$, we have
\begin{equation*}
    \begin{aligned}
       &\lim_{\eps\to 0}\frac{k_0^2z}{4\eta^2}\int\limits_{0}^1[R\big(\eta(y_{\gamma_2}-x_{\gamma_1})+\frac{\eta\eps^{\beta-1}(\xi_{\gamma_1}-\zeta_{\gamma_2})sz}{k_0}\big)-R(0)]\mathrm{d}s\\
       &=           \frac{k_0^2z}{8}\int\limits_{0}^1\Big(y_{\gamma_2}-x_{\gamma_1}+\frac{(\xi_{\gamma_1}-\zeta_{\gamma_2})sz}{k_0}\Big)^\tps\Xi\Big(y_{\gamma_2}-x_{\gamma_1}+\frac{(\xi_{\gamma_1}-\zeta_{\gamma_2})sz}{k_0}\Big)\mathrm{d}s\,.
    \end{aligned}
\end{equation*}
Using a similar argument as in the kinetic regime, we have that
\begin{eqnarray*}
    M_{p,p}(z,r,X,Y;T)=\begin{cases}
        \Big(M_{p,p}^\infty(z,\cdot,X,Y)\ast\Sp(\cdot;T)\Big)(r,\ldots,r), \quad {\beta=1}\\
        \sum\limits_{\pi_p}\prod\limits_{j=1}^pM_{1,1}(z,r,x_{j},y_{\pi_p(j)};t_{j},t'_{\pi_p(j)}),\quad {\beta=1,} \ \theta\to 0,
    \end{cases}
\end{eqnarray*}
    where $M_{p,p}^\infty$ is given by
    \begin{eqnarray}\label{eqn:M_pp_inf_diff}
          M_{p,p}^\infty(z,X',Y',X,Y)=\sum\limits_{\pi_p}\prod\limits_{j=1}^p {M}^\infty_{1,1}(z,x'_{j},x_j,y_{\pi_p(j)})\delta_0(x'_{j}-y'_{\pi_p(j)})\,,
    \end{eqnarray}
    with $M_{1,1}^\infty$ given by
    \begin{eqnarray*}
        M_{1,1}^\infty(z,r,x,y) =\int\limits_{\mathbb{R}^d}\exp\Big[\int\limits_{0}^1\frac{k_0^2z}{8}\Big(y-x+\frac{\zeta sz}{k_0}\Big)^\tps\Xi\Big(y-x+\frac{\zeta sz}{k_0}\Big)\mathrm{d}s\Big]e^{i\zeta\cdot r}\frac{\mathrm{d}\zeta}{(2\pi)^d}\,.
    \end{eqnarray*}
    $M_{1,1}$ in the case $\theta\to 0$ is given by~\eqref{eqn:M_11_diffusive} with $\beta=1$. 

This completes the convergence analysis of the moments $m_{p.q}^\eps$ to their limits \eqref{eq:Mppdiffusive} in Theorem \ref{thm:large_corr_diffusive}.   
\subsection{Concluding the proofs of Theorems~\ref{thm:large_corr_kinetic} and~\ref{thm:large_corr_diffusive}}

\begin{proof}[Proof of Theorems~\ref{thm:large_corr_kinetic} and~\ref{thm:large_corr_diffusive}]
   Corollary \ref{coro:Phi_mom} along with the limiting expressions in sections \ref{subsec:kinetic_mom} and \ref{subsec:diff_mom} show that the moments $M^\eps_{p,q}(\cdot)$ of $\Phi^\eps$ converge to their expression $M_{p,q}(\cdot)$ as given in Theorems \ref{thm:large_corr_kinetic} and \ref{thm:large_corr_diffusive}, respectively. Since the $p+p$th moments of such random vectors do not grow faster than $C^p p^{2p}$, their limiting distributions are completely characterized by their moments~\cite{billingsley2017probability}, which concludes the proofs of the theorems.
\end{proof}

\section{Limiting intensity distribution in the diffusive regime}
\label{sec:limiting_intensity}
From the results of section \ref{subsec:diff_mom}, we have that in the diffusive regime the intensity $I^\eps(z,r,x;t)=|\phi^\eps(z,r,x;t)|^{2}$ has limiting moments given by $M_{p,p}(z,r,x,\cdots,x;t,\cdots,t):=I_p(z,r)$ with $M_{p,p}$ as defined \eqref{eq:Mppdiffusive}. In particular, when $\beta>1$, we have that  
\begin{eqnarray}\label{eqn:I_p_alpha=beta>1_diff}
I_p(z,r)=\SpI(r,\ldots,r)p!=\Gamma(r,0)^p(p!)^2\,.
\end{eqnarray}
From the above expression for $M_{p,p},$ we have that when $\beta=1$, 
\begin{equation*}
\begin{aligned}
     I_p(z,r)& =\sum\limits_{\pi_p}\int\limits_{\mathbb{R}^{pd}}\Sp(X',X'_{\pi_p};0)\mathrm{d}x'_1\cdots\mathrm{d}x'_p\int\limits_{\mathbb{R}^{pd}}e^{\frac{z^3}{24}\sum_{j=1}^p\zeta_j^\tps\Xi\zeta_j}e^{i\sum_{j=1}^p\zeta_j\cdot (r-x'_j)}\frac{\mathrm{d}\zeta_1\cdots\mathrm{d}\zeta_p}{(2\pi)^{pd}}\\
&=p!\int\limits_{\mathbb{R}^{pd}}\SpI(X')\mathrm{d}x'_1\cdots\mathrm{d}x'_p\int\limits_{\mathbb{R}^{pd}}e^{\frac{z^3}{24}\sum_{j=1}^p\zeta_j^\tps\Xi\zeta_j}e^{i\sum_{j=1}^p\zeta_j\cdot (r-x'_j)}\frac{\mathrm{d}\zeta_1\cdots\mathrm{d}\zeta_p}{(2\pi)^{pd}}
\end{aligned}
\end{equation*}
where we have used the fact that
\begin{eqnarray}\label{eqn:mu_pp_symm}
    \Sp(X',X'_{\pi_p};0)=\sum\limits_{\pi'_p}\prod\limits_{j=1}^pJ_\theta(x_j',x'_{\pi'_p(\pi_p(j))})=\SpI(X')
\end{eqnarray}
 where $J_\theta$ is given by~\eqref{eqn:J_theta} and $\pi_p'$ is a permutation of $p$ integers independent of $\pi_p$. This translates to a diffusion equation of the form
\begin{equation*}
    I_p(z,r)=p![G_p(z^3,\cdot)\ast \SpI(\cdot)](r\cdots,r)\,,
\end{equation*}
where $G_p$ is the Green's function to the diffusion equation
\begin{equation}\label{eqn:high_dim_diff}
    \partial_tG_p+\frac{1}{24}\sum\limits_{j=1}^p\nabla_{r_j}\cdot\Big(\Xi\nabla_{r_j}G_p\Big)=0,\quad G_p(0,r_1,\cdots,r_p)=\prod\limits_{j=1}^p\delta_0(r_j)\,,
\end{equation}
and $\SpI(X)=\Sp(X,X;0)$ is given in \eqref{eqn:Sp}.
Again, from \eqref{eq:Mppdiffusive} we have that when $\beta=1$ and $\theta\to 0$,
\begin{eqnarray*}
    I_p(z,r)=p!M_{1,1}(z,r,x,x;t,t)^p=p!\mathbb{E}[I](z,r)^p\,,
\end{eqnarray*}
with $\mathbb{E}[I](z,r)$ given by~\eqref{eqn:E_I} with $\beta=1$. 

\begin{proof}[Proof of Corollary \ref{coro:intensity_large_corr}] 
This is a direct consequence of Theorem \ref{thm:large_corr_diffusive} and the above calculations.
\end{proof}

\subsection{Limiting time averaged intensity}
The limiting moments of the time averaged intensity are computed from \eqref{eq:Mppdiffusive} in a similar manner. Recall the time averaged intensity
    $I^\eps_{\sfT}(z,r,x;t)=\frac{1}{\sfT}\int_{0}^\sfT I^\eps(z,r,x;t+t')\mathrm{d}t'.$
    From Theorem~\ref{thm:large_corr_diffusive}, we have that
    \begin{equation*}
        \mathbb{E}[I^\eps_\sfT(z,r,x;t)^p]\to\frac{1}{\sfT^p}\int\limits_{[0,T]^p}M_{p,p}(z,r,x,\ldots,x;t+T,t+T)\prod\limits_{j=1}^p\mathrm{d}t_j\,,
    \end{equation*}
    as $\eps\to 0$ where $T=(t_1,\cdots,t_p)$.     
    When $\beta>1$, using \eqref{eq:Mppdiffusive} we have
    \begin{equation*}
    \begin{aligned}
       M_{p,p}(z,r,x,\ldots,x;t+T,t+T)=\Sp(r,\ldots,r;t+T,t+T)\sum\limits_{\pi_p}\prod\limits_{j=1}^pM_{1,1}^\infty(z,r,x,x)\\
       =\Big(\sum\limits_{\pi_p}\prod\limits_{j=1}^p\Gamma(r,0)M_{1,1}^\infty(z,r,x,x)\Big)\sum\limits_{\pi_p}\prod\limits_{j=1}^pF(t_j-t_{\pi_p(j)})
       =\frac{1}{p!}I_p(z,r)\sum\limits_{\pi_p}\prod\limits_{j=1}^pF(t_j-t_{\pi_p(j)})\,.
    \end{aligned}
        \end{equation*}
 Similarly, when $\beta=1$, we have
        \begin{equation*}
            \begin{aligned}
                 &M_{p,p}(z,r,x,\ldots,x;t+T,t+T)=\Big(M_{p,p}^\infty(z,\cdot,x,\ldots,x)\ast \Sp(\cdot,t+T,t+T)\Big)(r,\ldots,r)\\
                 &=\sum\limits_{\pi_p}\prod_{j=1}^pF(t_j-t_{\pi_p(j)})\int\limits_{\mathbb{R}^{pd}}\sum\limits_{\pi'_p}\prod\limits_{j=1}^pJ_\theta(x'_j,x'_{\pi'_p(\pi_p(j))})\Big[\int\limits_{\mathbb{R}^{pd}}e^{\frac{z^3}{24}\sum_{j=1}^p\zeta_j^\tps\Xi\zeta_j}e^{i\sum_{j=1}^p\zeta_j\cdot(r-x'_j)}\frac{\mathrm{d}\zeta_1\cdots\mathrm{d}\zeta_p}{(2\pi)^{pd}}\Big]\mathrm{d}x'_1\cdots\mathrm{d}x'_p\\
                 &=\frac{1}{p!}I_p(z,r)\sum\limits_{\pi_p}\prod\limits_{j=1}^pF(t_j-t_{\pi_p(j)})
            \end{aligned}
        \end{equation*}
        where we have used~\eqref{eqn:mu_pp_symm} again in the last line. A similar argument holds when $\beta=1$ and $\theta\to 0$ as well.
\begin{proof}[Proof of Theorem \ref{thm:time_avg_intens_limit}.]
    Recalling the definition of $F_p(\sfT)$ in 
    \eqref{eqn:F_p(T)}, we deduce from Corollary \ref{coro:intensity_large_corr} and the above expressions that \begin{equation*}
        \mathbb{E}[I^\eps_\sfT(z,r,x;t)^p]\to\frac{1}{p!}I_p(z,r)F_p(\sfT)\,,
    \end{equation*}
    as $\eps\to 0$. An application of Theorem \ref{thm:large_corr_diffusive} concludes the proof. 
\end{proof}

\subsection{Time averaged scintillation and proof of Corollary~\ref{coro:scint_large_correl}}
 Theorem~\ref{thm:time_avg_intens_limit} gives
\begin{equation*}
    \begin{aligned}
      \mathbb{E}[I^\eps_\sfT(z,r,x;t)^2]&\to\frac{1}{2}\mathbb{E}[I(z,r)^2]F_2(\sfT)=\frac{1}{2}\mathbb{E}[I(z,r)^2](1+F_\sfT)\,.
         \end{aligned}
\end{equation*}
Note that for $\beta>1$, we deduce from Corollary~\ref{coro:intensity_large_corr} that
\begin{equation*}
    \begin{aligned}
    \mathbb{E}[I(z,r)^2]=4\Gamma(r,0)^2=4\mathbb{E}[I(z,r)]^2\,.
    \end{aligned}
\end{equation*}
This gives the limiting scintillation when $\beta>1$ as
\begin{eqnarray*}
    \sfS_\sfT(z,r)=1+2F_\sfT\,.
\end{eqnarray*}
Similarly, for $\beta=1$ we have that the limiting scintillation is given by
\begin{eqnarray*}
    \sfS_\sfT(z,r)=F_\sfT+\frac{\chi(z,r)}{\mathbb{E}[I(z,r)]^2}(1+F_\sfT)\,,
\end{eqnarray*}
where $\chi(z,r)=\chi(z,r;\theta)$ denotes the cross correlation component 
\begin{equation}\label{eqn:chi_eq1}
    \begin{aligned}
       \chi(z,r)
       &=\int\limits_{\mathbb{R}^{4d}}
       e^{i\xi\cdot(r-x)}e^{i\zeta\cdot(r-y)}e^{\frac{z^3}{24}\big(\xi^\tps\Xi\xi+\zeta^\tps\Xi\zeta\big)}J_\theta(x,y)J_\theta(y,x)\frac{\mathrm{d}\xi\mathrm{d}\zeta\mathrm{d}x\mathrm{d}y}{(2\pi)^{2d}}\\
       &=\Big(\frac{12}{z^3}\Big)^d\frac{1}{|\Xi|}\int\limits_{\mathbb{R}^{2d}}e^{\frac{6}{z^3}\big((r-x)^\tps\Xi^{-1}(r-x)+(r-y)^\tps\Xi^{-1}(r-y)\big)}J_\theta(x,y)J_\theta(y,x)\frac{\mathrm{d}x\mathrm{d}y}{(2\pi)^d}\,,
    \end{aligned}
\end{equation}
where $|\Xi|$ denotes the determinant of $\Xi$ and $J_\theta$ is given by~\eqref{eqn:J_theta}. From a change of variables $\big(\frac{x+y}{2},x-y\big)\to (r',\theta\sigma')$, we have that
\begin{equation*}
    \begin{aligned}
         \chi(z,r)&=\Big(\frac{12\theta}{z^3}\Big)^d\frac{1}{|\Xi|}\int\limits_{\mathbb{R}^{2d}}e^{\frac{12}{z^3}(r-r')^\tps\Xi^{-1}(r-r')}e^{\frac{3\theta^2}{z^3}\sigma'^\tps\Xi^{-1}\sigma'}\Gamma(r',\sigma')^2\frac{\mathrm{d}r'\mathrm{d}\sigma'}{(2\pi)^d}\,.
    \end{aligned}
\end{equation*}
In particular, when $\theta\to 0$, from Lebesgue's dominated convergence theorem we have that
\begin{equation*}
    \chi(z,r)\to 0\,.
\end{equation*}
Note that the difference $\chi(z,r)-\mathbb{E}[I](z,r)^2$ follows a diffusion equation
\begin{equation*}
    \chi(z,r)-\mathbb{E}[I](z,r)^2=\big(G_2(z^3,\cdot,\cdot)\ast S\big)(r,r)\,,
\end{equation*}
where $G_2$ is given by~\eqref{eqn:high_dim_diff} and $S(r_1,r_2)=J_\theta(r_1,r_2)J_\theta(r_2,r_1)-J_\theta(r_1,r_1)J_\theta(r_2,r_2)$. By Cauchy-Schwarz,
\begin{equation*}
   S(r_1,r_2)=J_\theta(r_1,r_2)J_\theta(r_2,r_1)-J_\theta(r_1,r_1)J_\theta(r_2,r_2)=|\mathbb{E}[u_0(r_1)u_0^\ast(r_2)]|^2 -\mathbb{E}[|u_0(r_1)|^2]\mathbb{E}[|u_0(r_2)|^2]\le 0\,.
\end{equation*}
This means that $\chi(z,r)-\mathbb{E}[I](z,r)^2$ is given by the solution of a diffusion equation with a source that is strictly non positive, evaluated along the diagonal $r_1=r_2=r$. Since the solution to the heat equation is simply a convolution with a Gaussian, the non positivity due to the initial condition is preserved at all times. Furthermore, $\mathbb{E}[I](z,r)$ follows a diffusion equation given by~\eqref{eqn:E_I_large_correl_beta=1}, \eqref{eqn:G_large_correl_beta=1}. So it stays positive if $\Gamma(r,0)>0$. This gives us
\begin{equation*}
    \frac{\chi(z,r)}{\mathbb{E}[I](z,r)^2}\le 1\,.
\end{equation*}
This proves~\eqref{eqn:E_I^2_large_corr}-\eqref{eqn:Scint_large_corr}. 
\paragraph{Limiting behavior of $\frac{\chi(z,r)}{\mathbb{E}[I](z,r)^2}$ as $z\to\infty$. } This ratio is given by
\begin{equation*}
    \frac{\chi(z,r)}{\mathbb{E}[I](z,r)^2}=\frac{\int\limits_{\mathbb{R}^{2d}}e^{\frac{6}{z^3}\big((r-x)^\tps\Xi^{-1}(r-x)+(r-y)^\tps\Xi^{-1}(r-y)\big)}J_\theta(x,y)J_\theta(y,x)\mathrm{d}x\mathrm{d}y}{\int\limits_{\mathbb{R}^{2d}}e^{\frac{6}{z^3}\big((r-x)^\tps\Xi^{-1}(r-x)+(r-y)^\tps\Xi^{-1}(r-y)\big)}J_\theta(x,x)J_\theta(y,y)\mathrm{d}x\mathrm{d}y}\,.
\end{equation*}
From Lebesgue's dominated convergence theorem,
\begin{equation*}
    \lim_{z\to\infty}\frac{\chi(z,r)}{\mathbb{E}[I](z,r)^2}=\frac{\int\limits_{\mathbb{R}^{2d}}J_\theta(x,y)J_\theta(y,x)\mathrm{d}x\mathrm{d}y}{\Big(\int\limits_{\mathbb{R}^{d}}J_\theta(x,x)\mathrm{d}x\Big)^2}=\frac{\int\limits_{\mathbb{R}^{2d}}\Gamma\big(r,\frac{\sigma}{\theta}\big)^2\mathrm{d}r\mathrm{d}\sigma}{\Big(\int\limits_{\mathbb{R}^{d}}\Gamma(r,0)\mathrm{d}r\Big)^2}\,.
\end{equation*}
This shows~\eqref{eqn:Scint_z_infty} and completes the proof.
\section{Conclusions and possible extensions}
\label{sec:conclusion}
The goal of this paper was to provide a complete statistical description of wavefields due to partially coherent incident beams under the scintillation scaling of the It\^o-Schr\"odinger white noise paraxial regime. We observe that statistical fluctuations in the time averaged intensity are amplified or reduced depending on the coherence width of the source relative to the fluctuations in the medium and the averaging time at the detector. In particular, this leads to scintillation indices ranging from 3 to 0 in the diffusive limit. The results presented in this paper can potentially be useful in providing a theoretical guidance to optimize performance of communication links. The diffusive regime considered here provides a simplified mapping between the turbulence parameters of the medium, source statistics and intensity measurements. The asymptotic statistical behaviour of wavefields in this regime is primarily governed by the macroscopic scaling parameters rather than the intricate profiles of the incident beams themselves. While the kinetic regime does not immediately lead to as simplified models as in the diffusive regime, it may prove useful for engineering applications when complemented with numerical simulations. \textcolor{black}{It is also likely that one can arrive at a similar conclusion starting from the paraxial approximation, without explicitly using the white noise model as was done in~\cite{bal2024long} for coherent beams.} Quantifying the range of validity of these regimes and further numerical experiments in the kinetic regime are left for future investigation. 

\section*{Disclosure statement} The authors report there are no competing interests to declare. 

\section*{Acknowledgements} This work was supported in part by NSF Grant DMS-2306411.
\appendix
\section{Operator \texorpdfstring{$L^\eps_{p,q}$}{Lpq}}\label{App:operator}
The operator $L^\eps_{p,q}$ is defined as
\begin{equation}\label{eqn:L_expansion}
     {L}^{\eps}_{p,q}=\frac{p+q}{2}L_{\eta}+\sum_{j=1}^p\sum_{l=1}^q{L}^{\eps,1}_{j,l}+\sum_{1\le j< j'\le p}L^{\eps ,2}_{j,j'}+\sum_{p+1\le l< l'\le p+q}L^{\eps ,2}_{l,l'}\,,
 \end{equation}
  with   
 \begin{equation}\label{eqn:L_pq_def}
     \begin{aligned}
          L_{\eta}&=-\frac{k_0^2}{4(2\pi)^d\eta^2}\int_{\Rm^d}\hat{R}(k)\mathrm{d}k=-\frac{k_0^2 R(0)}{4\eta^2}\,,\\
         {L}^{\eps ,1}_{j,l}\rho(v)&=\frac{k_0^2}{4(2\pi)^d\eta^2}\int_{\Rm^d}\hat{R}(k)\rho(v-A_{j,l}k)e^{\frac{iz\eta}{k_0\eps}k^\tps A_{j,l}^\tps\Theta v}\mathrm{d}k\,,\\
         L^{\eps ,2}_{j,j'}\rho(v)&=-\frac{k_0^2}{4(2\pi)^d\eta^2}\int_{\Rm^d}\hat{R}(k)\rho(v-B_{j,j'}k)e^{\frac{iz\eta}{2k_0\eps}(2k^\tps B_{j,j'}^\tps\Theta v-k^\tps B_{j,j'}^\tps\Theta B_{j,j'}k)}\mathrm{d}k\,.
     \end{aligned}
 \end{equation}
For $1\le j< j'\le p$ and $p+1\le j< j'\le p+q$, we define $B_{j,j'}$ as a $d(p+q)\times d$ matrix of $d\times d$ blocks with block $j$ as $\mathbb{I}_d$ and block $j'$ as $-\mathbb{I}_d$.
\section{Stochastic continuity and tightness}\label{App:tightness}
\begin{proof}[Proof of Theorem \ref{thm:tightness}]
For fixed $z>0$ and $r\in\mathbb{R}^d$ we show a tightness criterion of the form
\begin{equation*}
     \sup_{s\in[0,z]} \mathbb{E}|\phi^\eps(s,r,x+h;t+\Delta t)-\phi^\eps(s,r,x;t)|^{2n}\le C_{\alpha}(z,n)(|h|^{2\alpha_0 n}+|\Delta t|^{2n})\,.
\end{equation*}
For this, it is sufficient to show that 
\begin{eqnarray*}
    \sup_{s\in[0,z]}\mathbb{E}|\phi^\eps(s,r,x+h;t)-\phi^\eps(s,r,x;t)|^{2n}\le C_\alpha(n,z)|h|^{2\alpha_0 n}\,,
\end{eqnarray*}
uniformly in $t$ and
\begin{eqnarray*}
    \sup_{s\in[0,z]}\mathbb{E}| \phi^\eps(s,r,x;t+\Delta t)-\phi^\eps(s,r,x;t)|^{2n}\le C(z,n)|\Delta t|^{2 n}\,,
\end{eqnarray*}
uniformly in $x$.
\paragraph{Stochastic continuity along $x$:}

For deterministic sources, in order to have 
\begin{eqnarray*}
    \sup_{s\in[0,z]}\mathbb{E}|\phi^\eps(s,r,x+h)-\phi^\eps(s,r,x)|^{2n}\le C_\alpha(z,n)|h|^{2\alpha_0 n}\,,
\end{eqnarray*}
it is sufficient that $\|\langle k\rangle^{2}\hat{u}_0(k)\|\le C$ (upon some additional assumptions on the medium, see Theorem 2.7 in~\cite{bal2024complex}). In particular, $\sup_{s\in[0,z]}\mathbb{E}|\phi^\eps(s,r,x+h)-\phi^\eps(s,r,x)|^{2n}\le C_\alpha' \|\langle k\rangle^{2}\hat{u}_0(k)\|^{2n}|h|^{2\alpha_0n}$. For randomly varying sources, for fixed $(r,t)$ we have
\begin{equation*}
    \mathbb{E}|\phi^\eps(s,r,x+h;t)-\phi^\eps(s,r,x;t)|^{2n}=\mathbb{E}\big[\mathbb{E}[|\phi^\eps(s,r,x+h;t)-\phi^\eps(s,r,x;t)|^{2n}\big|u_0]\big]\le C'_\alpha|h|^{2\alpha_0n}\mathbb{E}[ \|\langle k\rangle^{2}\hat{u}_0(k;t)\|^{2n}]\,.
\end{equation*}
By the Cauchy-Schwarz inequality,
\begin{equation*}
    \begin{aligned}
        \|\langle k\rangle^2\hat{u}_0(k;t)\|^2&=\Big(\int\limits_{\mathbb{R}^d}\langle k\rangle^{2+\alpha/2}|\hat{u}_0(k;t)|\frac{1}{\langle k\rangle^{\alpha/2}}\mathrm{d}k\Big)^2\le \int\limits_{\mathbb{R}^d}\frac{1}{\langle k\rangle^\alpha}\mathrm{d}k\int\limits_{\mathbb{R}^d}\langle k\rangle^{4+\alpha}|\hat{u}_0(k;t)|^2\mathrm{d}k\,.
    \end{aligned}
\end{equation*}
Using the Gaussian statistics of the source, this gives
\begin{equation*}
    \begin{aligned}
        \mathbb{E}[\|\langle k\rangle^2\hat{u}_0(k;t)\|^{2n}]&\le \Big(\int\limits_{\mathbb{R}^d}\frac{1}{\langle k\rangle^{2\alpha}}\mathrm{d}k\Big)^{n}\int\limits_{\mathbb{R}^{nd}}\prod\limits_{j=1}^n\langle k_j\rangle^{4+\alpha}\mathbb{E}[\prod\limits_{j=1}^n|\hat{u}_0(k_j;t)|^2]\mathrm{d}k_1\cdots\mathrm{d}k_n\\
        &= \Big(\int\limits_{\mathbb{R}^d}\frac{1}{\langle k\rangle^{2\alpha}}\mathrm{d}k\Big)^{n}\sum\limits_{\pi_n}\int\limits_{\mathbb{R}^{nd}}\prod\limits_{j=1}^n\langle k_j\rangle^{2+\alpha/2}\langle k_{\pi_n(j)}\rangle^{2+\alpha/2}\hat{J}(k_j,k_{\pi_n(j)})F(0)\mathrm{d}k_1\cdots\mathrm{d}k_n\,.
    \end{aligned}
\end{equation*}
For the first term to be bounded, we will need that $\alpha>\frac{d-1}{2}$. We have that the expectation on the left hand side is bounded uniformly in $t$ as soon as for every $\pi_n$,
\begin{equation*}
    \int\limits_{\mathbb{R}^{nd}} \prod\limits_{j=1}^n\langle k_j\rangle^{2+\alpha/2}\langle k_{\pi_n(j)}\rangle^{2+\alpha/2}\hat{J}(k_j,k_{\pi_n(j)})\mathrm{d}k_1\cdots\mathrm{d}k_n\le C(n,\alpha)\,.
\end{equation*}
\paragraph{Stochastic continuity along $t$:}

Now define 
\begin{eqnarray*}
    \delta_tf=f(t+\delta t)-f(t)\,.
\end{eqnarray*}
Assuming that the medium does not vary with time, $\delta_t u^\eps$ follows the same SDE~\eqref{eqn:Ito_Schr} as $u^\eps$ with source
\begin{eqnarray*}
    \delta_t u^\eps(0,x;t)=u_0^\eps(x;t+\Delta t)-u_0^\eps(x;t)\,.
\end{eqnarray*}
From Proposition~\ref{prop:N_approx_phy}, we have that the $p+p$th moments of $\delta_tu^\eps$ are bounded independently of $\eps$:
\begin{equation*}
    \|\mathbb{E}\prod_{j=1}^p\delta_tu^\eps(z,x;t_j)\prod_{l=1}^p\delta_tu^{\eps *}(z,x;t'_l)\|_\infty \le C(z,p)\|\mathbb{E}\prod_{j=1}^p\widehat{\delta_tu^\eps}(0,\xi_j;t_j)\prod_{l=1}^p\widehat{\delta_tu^{\eps *}}(0,\zeta_l;t'_l)\|
\end{equation*}
Due to this, we note that
\begin{eqnarray*}
    \sup_{s\in[0,z]}\mathbb{E}| \phi^\eps(s,r,x;t+\Delta t)-\phi^\eps(s,r,x;t)|^{2n}\le C(z,n)|\Delta t|^{2 n}\,,
\end{eqnarray*}
holds true if
\begin{eqnarray}\label{eqn:stoch_cont_t}
    \|\mathbb{E}[\prod_{j=1}^p\big(\hat{u}^\eps_0(\xi_j;t_j+\Delta t)-\hat{u}^\eps_0(\xi_j;t_j)\big)\prod_{l=1}^p\big(\hat{u}^{\eps \ast}_0(\zeta_l;t'_l+\Delta t)-\hat{u}^{\eps \ast}_0(\zeta_l;t'_l)\big)]\|\le C(z,n) |\Delta t|^{2p}\,.
\end{eqnarray}
Define the Fourier transform
$
    \check{u}^\eps_0(\xi;\omega) =\int\limits_{\mathbb{R}}\hat{u}^\eps_0(\xi;t)e^{-i\omega t}\mathrm{d}t\,.
$
Then we have
\begin{equation*}
    \begin{aligned}
        &\mathbb{E}[\prod_{j=1}^p\big(\hat{u}^\eps_0(\xi_j;t_j+\Delta t)-\hat{u}^\eps_0(\xi_j;t_j)\big)\prod_{l=1}^p\big(\hat{u}^{\eps \ast}_0(\zeta_l;t'_l+\Delta t)-\hat{u}^{\eps \ast}_0(\zeta_l;t'_l)\big)]\\
        &=\int\limits_{\mathbb{R}^{2p}}\mathbb{E}[\prod\limits_{j=1}^p\check{u}^\eps_0(\xi_j;\omega_j)\prod\limits_{l=1}^p\check{u}^{\eps \ast}_0(\zeta_l;\omega'_l)]e^{i\sum_{j=1}^p(\omega_jt_j-\omega'_jt'_j)}\prod\limits_{j=1}^p(e^{i\omega_j\Delta t}-1)(e^{-i\omega'_j\Delta t}-1)\prod\limits_{j=1}^p\frac{\mathrm{d}\omega_j\mathrm{d}\omega'_j}{(2\pi)^2}
    \end{aligned}
\end{equation*}
Using~\eqref{eqn:u0_cov} we have that
\begin{eqnarray*}
    \mathbb{E}[\check{u}^\eps_0(\xi;\omega)\check{u}_0^{\eps \ast}(\zeta;\omega')]=2\pi\delta_0(\omega-\omega')\hat{F}(\omega)\hat{J}^\eps(\xi,\zeta)\,,
\end{eqnarray*}
which along with the Gaussian assumption on $u_0^\eps$ gives
\begin{equation*}
    \begin{aligned}
        &\mathbb{E}[\prod_{j=1}^p\big(\hat{u}^\eps_0(\xi_j;t_j+\Delta t)-\hat{u}^\eps_0(\xi_j;t_j)\big)\prod_{l=1}^p\big(\hat{u}^{\eps \ast}_0(\zeta_l;t'_l+\Delta t)-\hat{u}^{\eps \ast}_0(\zeta_l;t'_l)\big)]\\
        &=\sum\limits_{\pi_p}\prod\limits_{j=1}^p\hat{J}^\eps(\xi_j,\zeta_{\pi_p(j)})\int\limits_{\mathbb{R}^{p}}e^{i\sum_{j=1}^p\omega_j(t_j-t'_{\pi_p(j)})}\prod\limits_{j=1}^p\hat{F}(\omega_j)|e^{i\omega_j\Delta t}-1|^2\frac{\mathrm{d}\omega_1\cdots\mathrm{d}\omega_p}{(2\pi)^p}\,.
    \end{aligned}
\end{equation*}
This gives
\begin{equation*}
    \begin{aligned}
        &|\mathbb{E}[\prod_{j=1}^p\big(\hat{u}^\eps_0(\xi_j;t_j+\Delta t)-\hat{u}^\eps_0(\xi_j;t_j)\big)\prod_{l=1}^p\big(\hat{u}^{\eps \ast}_0(\zeta_l;t'_l+\Delta t)-\hat{u}^{\eps \ast}_0(\zeta_l;t'_l)\big)]|\\
        &\le(\Delta t)^{2p}\Big(\int\limits_{\mathbb{R}}\omega^2|\hat{F}(\omega)|\frac{\mathrm{d}\omega}{2\pi}\Big)^{p}\sum\limits_{\pi_p}\prod\limits_{j=1}^p|\hat{J}^\eps(\xi_j,\zeta_{\pi_p(j)})|\,.
    \end{aligned}
\end{equation*}
So if we have that
$\int\limits_{\mathbb{R}}\omega^2|\hat{F}(\omega)|\mathrm{d}\omega\le C$
along with the integrability of $\hat{J}^\eps$, we obtain that \eqref{eqn:stoch_cont_t} holds.
\end{proof}

\bibliographystyle{siam}
\bibliography{Reference}

\end{document}